\documentclass{amsart}
\usepackage{latexsym,amsxtra,amscd,ifthen}
\usepackage{amsfonts}
\usepackage{amsmath}
\usepackage{amsthm}
\usepackage{amssymb}
\usepackage{mathtools}
\usepackage[colorlinks]{hyperref}
\numberwithin{equation}{section}
\usepackage{ulem}
\theoremstyle{plain}

\usepackage{color}
\definecolor{darkgreen}{RGB}{55,138,0}
\hypersetup{citecolor = darkgreen}

\newtheorem{theorem}{Theorem}[section]
\newtheorem{lemma}[theorem]{Lemma}
\newtheorem{proposition}[theorem]{Proposition}
\newtheorem{corollary}[theorem]{Corollary}

\theoremstyle{definition}
\newtheorem{definition}[theorem]{Definition}
\newtheorem{example}[theorem]{Example}

\newtheorem{remark}[theorem]{Remark}
\newtheorem{question}[theorem]{Question}
\newtheorem{hypothesis}[theorem]{Hypothesis}

\numberwithin{equation}{theorem} 

\newcommand{\bfl}{\mathfrak l}

\DeclareMathOperator{\hdet}{hdet} 
\DeclareMathOperator{\GL}{GL}
\DeclareMathOperator{\gldim}{gldim} 

\DeclareMathOperator{\Ext}{Ext} 
\DeclareMathOperator{\tr}{tr}
\DeclareMathOperator{\GKdim}{GKdim}
\DeclareMathOperator{\End}{End} 
\DeclareMathOperator{\p}{{\sf p}} 
\DeclareMathOperator{\Gal}{Gal}

\newcommand{\id}{\operatorname{id}}

\newcommand{\ch}{\operatorname{char}}

\newcommand\inv{^{-1}}
\newcommand\kk{\Bbbk}

\DeclareMathOperator{\adj}{adj}
\DeclareMathOperator{\Aut}{Aut}
\newcommand{\Autgr}{\Aut_{\operatorname{gr}}}
\DeclareMathOperator\diag{diag}
\DeclareMathOperator\Id{Id}

\DeclareMathOperator\Oz{Oz}
\DeclareMathOperator{\Padj}{Padj}
\DeclareMathOperator\pf{pf}
\DeclareMathOperator{\rk}{rk}

\DeclareMathOperator\SL{SL}
\DeclareMathOperator\Tr{Tr}

\newcommand\be{\mathbf e}
\newcommand\bp{\mathbf p}
\newcommand\bq{\mathbf q}
\newcommand\bu{\mathbf u}
\newcommand\bv{\mathbf v}
\newcommand\bx{\mathbf x}

\newcommand\by{\mathbf y}

\newcommand{\fa}{\mathfrak{a}}
\newcommand{\fb}{\mathfrak{b}}
\newcommand{\mff}{\mathfrak{f}}
\newcommand{\fj}{\mathfrak{j}}

\newcommand\oj{\mathfrak{oj}}
\newcommand\oa{\mathfrak{oa}}
\newcommand\od{\mathfrak{od}}
\newcommand\pg{\mathfrak{pg}}

\newcommand\ZZ{\mathbb Z}
\newcommand{\SP}{S_{\bp}}
\newcommand{\ZSP}{ZS_{\bp}}

\newcommand\restrict[1]{\raisebox{-.3ex}{$|$}_{#1}}

\begin{document}

\title[Ozone groups and centers of skew polynomial rings]
{Ozone groups and centers of skew polynomial rings}

\author{Kenneth Chan, Jason Gaddis, Robert Won, James J. Zhang}

\address{(Chan) Department of Mathematics, Box 354350,
University of Washington, Seattle, Washington 98195, USA}
\email{kenhchan@math.washington.edu, ken.h.chan@gmail.com}

\address{(Gaddis) Department of Mathematics, 
Miami University, Oxford, Ohio 45056, USA} 
\email{gaddisj@miamioh.edu}

\address{(Won) Department of Mathematics,
George Washington University, Washington, DC 20052, USA}
\email{robertwon@gwu.edu}

\address{(Zhang) Department of Mathematics, Box 354350,
University of Washington, Seattle, Washington 98195, USA}
\email{zhang@math.washington.edu}

\begin{abstract}
We introduce the ozone group of a noncommutative algebra $A$, 
defined as the group of automorphisms of $A$ which fix every 
element of its center. In order to initiate the study of ozone 
groups, we study PI skew polynomial rings, which have long 
proved to be a fertile testing ground in noncommutative 
algebra. Using the ozone group and other invariants defined 
herein, we give explicit conditions for the center of a PI 
skew polynomial ring to be Gorenstein (resp. regular) in low 
dimension.
\end{abstract}



\maketitle

\setcounter{section}{-1}
\section{Introduction}
\label{zzsec0}

The centers of PI Artin--Schelter regular algebras of global 
dimension three, including the skew polynomial rings of three 
variables, have been studied by Artin \cite{Art} and Mori 
\cite{Mori} from the point of view of noncommutative projective 
geometry. A motivation of this paper is to study the center of 
a PI skew polynomial ring of any number of variables from an 
algebraic point of view. 

One of our aims is to search for new invariants that control 
the structure of a PI skew polynomial ring and its center, as 
well as the interplay between them. For example, are there 
combinatorial invariants that are closely related to the 
Calabi--Yau property of a PI skew polynomial ring or the 
Gorenstein property of its center? In our results below, we 
provide several invariants which provide an affirmative answer 
to this question (e.g., $\pg_S$ in Theorem~\ref{zzthm0.10} or 
the element $\oj_S\; \pg_S$ in Theorem~\ref{zzthm0.11}).

We start with a more general class of algebras, as we hope 
that some of the ideas in this paper will work more generally. 
Throughout, let $\kk$ denote a base field with $\ch \kk = 0$. 
Let $A$ be a noetherian PI Artin--Schelter regular algebra 
with center $Z$. The first combinatorial/algebraic invariant 
we introduce is the following, which will be studied in more 
detail in \cite{CGWZ2}.

\begin{definition}[{\cite[Definition 1.1]{CGWZ2}}]
\label{zzdef0.1} 
The {\it ozone group} of $A$ is
\[\Oz(A):=\Aut_{Z{\text{-alg}}}(A).\]
\end{definition}

The observant reader will notice a similarity between the ozone group and the Galois group of a field extension $E \subset F$. The classical Galois group $\Gal(F/E)$ is not an ozone group, \emph{per se}, as the ozone group of any commutative ring is trivial. However, both Galois groups and ozone groups involve automorphisms of an overstructure which fix a central substructure. This relationship will be considered further in \cite{CGWZ2}.

By \cite[Theorem 1.8]{CGWZ2}, the order of the ozone group 
satisfies the following condition
\[ 1\leq |\Oz(A)|\leq \rk(A_Z).\]
It is quite surprising that the skew polynomial rings play a 
role at one of the extreme cases of the above inequalities 
and that they can, in fact, be characterized in terms of 
properties of the ozone group [Lemma~\ref{zzlem0.2}]. 
 
Let $\bp := (p_{ij}) \in M_n(\kk)$ be a multiplicatively 
antisymmetric matrix. The \emph{skew polynomial ring} 
\[\SP := \kk_{\bp}[x_1,\dots,x_n]\] 
is the $\kk$-algebra generated by $\{x_1,\dots,x_n\}$ and 
subject to relations
\[x_j x_i= p_{ij} x_i x_j \text{ for all } 1 \leq i, j \leq n.\]
By \cite[Theorem 2]{Ga}, the parameter ${\bf p}$, up to a 
permutation of $\{1,\dots,n\}$, is an algebra invariant of 
$\SP$. It is easy to see that $\SP$ is PI (that is, $\SP$ 
satisfies a polynomial identity) if and only if each $p_{ij}$ 
is a root of unity and this is the setting we consider in 
this paper. 

Skew polynomial rings have been studied extensively and
their ring theoretic properties are well-understood. In particular, $\SP$ is an Artin--Schelter 
regular algebra of global and Gelfand--Kirillov dimension $n$. 
Skew polynomial rings provide a good set of examples to test theories related to Artin-Schelter regular algebras in general.

However, there are still unsolved questions concerning skew
polynomial rings. For example, 
when $p \neq 1$ is a root of unity 
it is unknown if $\Bbbk_{p}[x_1,x_2,x_3]$
is cancellative \cite[Question 0.8]{TRZ}
and it is not clear how to describe its full automorphism group.
In the current paper, we pay more attention to some homological questions.

\begin{lemma}[{\cite[Theorem 0.3]{CGWZ2}}]
\label{zzlem0.2} 
Suppose $\kk$ is algebraically closed. Let $A$ be a 
noetherian PI Artin--Schelter regular algebra generated 
in degree 1. Then $A$ is isomorphic to a skew polynomial 
ring if and only if $\Oz(A)$ is abelian and 
$|\Oz(A)|=\rk (A_Z)$. 
\end{lemma}

We note that the hypothesis on $\kk$ is not necessary for the forward direction in Lemma~\ref{zzlem0.2}.

For most of the results in this paper, we will assume the following hypothesis.

\begin{hypothesis}\label{zzhyp0.3}
Let $S = \SP = \kk_{\bp}[x_1,\ldots,x_n]$ be a PI skew 
polynomial ring with $\deg(x_i)=1$ for all $i$ and let 
$Z=Z\SP$ denote the center of $S$. Let $\xi$ be a 
primitive $\ell$th root of unity such that $p_{ij}=
\xi^{b_{ij}}$ for some integer $b_{ij}$. By convention, 
we choose the $b_{ij}$ such that $b_{ii}=0$ and 
$b_{ji}=-b_{ij}$ for all $i$ and $j$. We assume that 
$\ell$ is minimal and that $\ell>1$ (which is equivalent 
to the noncommutativity of $S$). Hence 
$\gcd(b_{ij} , \ell)_{1 \leq i,j \leq n} = 1$. Implicitly 
we fix a set of generators $\{x_1,\dots,x_n\}$. 

We let $\overline{\ZZ}$ denote $\ZZ/\ell\ZZ$ and let 
$\overline{a}$ (where $a\in\ZZ$, $\ZZ^n$, or $M_n(\ZZ)$) 
denote the image of $a$ modulo $\ell$. Let $B = (b_{ij})$, 
so $B$ and $\overline{B}$ are $n \times n$ skew-symmetric matrices.

Let $\phi_i$ denote the automorphism of $S$ given by 
conjugation by $x_i$, that is
\[ \phi_i(f) = x_i\inv f x_i \text{ for all } f \in S\]
and let $O$ be the subgroup of $\Autgr(S)$ generated by 
$\{\phi_1, \dots, \phi_n\}$. 
\end{hypothesis}

It turns out that $O$ is the ozone group $\Oz(S)$ 
\cite{CGWZ2}. For each fixed $i$, we have 
$\phi_i(x_s)=\xi^{b_{is}} x_s=p_{is}x_s$ 
for all $s$ and the order of $\phi_i$ is
\[ o(\phi_i)=\ell/\gcd\{b_{1i},\dots,b_{ni},\ell\}.\]
It is easy to see that the center of $\SP$ is 
\[ \ZSP=\SP^O.\]
Hence, we are able to use tools of noncommutative invariant 
theory to reveal properties of the center $\ZSP$. 

In general, it is difficult to discern properties of $\SP$ 
simply by examining its parameters ${\bf p}:= (p_{ij})$. 
However, in low-dimension ($n\leq 4$), many such properties 
can be worked out explicitly. In particular, we are interested 
in studying the center $\ZSP$ in terms of ${\bf p}$.

Note that most of the questions we consider are easily 
addressed in the case $n=2$. For a large part of this paper, 
we focus on the cases $n=3$ and $n=4$. We notice that there 
are some similarities and differences between the cases of 
$n$ being odd or even. In some cases we are able to say 
something about the case of higher $n$, and we hope that 
some of the ideas and results presented here can be extended 
to general $n$.

Several of our results make use of the \emph{Pfaffian} of the 
matrix $B$ associated to $\SP$, which we denote by $\pf(B)$. 
The Pfaffian was introduced by Cayley \cite{Cay}, based on 
earlier work of Jacobi \cite{Jac}. In particular, for even $n$, 
Cayley showed that the Pfaffian of a skew-symmetric matrix is 
$\sqrt{\det(B)}$. When $n$ is odd, every $n \times n$ matrix 
has Pfaffian $0$. Now suppose $n$ is even and let $A=(a_{ij})$ 
be a skew symmetric $n \times n$ matrix. Denote by 
$A_{\hat{i}\hat{j}}$ the submatrix of $A$ obtained by deleting 
both the $i$th row and column and the $j$th row and column. The 
Pfaffian of $A$ may be computed as
\[ \pf(A)=\sum_{k=2}^{n} (-1)^k a_{1k} \pf(A_{\hat{1}\hat{k}}).\]

Next we define a few more combinatorial invariants. For the 
remainder of this introduction, we assume 
Hypothesis~\ref{zzhyp0.3}. In particular, we assume $S=\SP$ 
is a PI skew polynomial ring with center $Z=\ZSP$. Define 
\begin{align}\label{E0.3.1}
\mff_i = \gcd\{d_i \mid \exists~d_1,\dots,\widehat{d_i},\dots,d_n ~\text{with}~ x_1^{d_1}\cdots x_i^{d_i}\cdots x_n^{d_n}\in Z\}.
\end{align}
By Proposition~\ref{zzpro1.8}(2), the set of integers 
$\{\mff_1,\dots,\mff_n\}$ with multiplicities, is an 
algebra invariant of $S$. 

In this paper we would like to show that $\{f_1,\cdots,f_n\}$
also play an essential role in several properties of
the algebra $\SP$. They 
have been used to control the automorphism group of $S$. 
For example, if $\mff_i\geq 2$ for all $i$, then 
every automorphism of $S$ is affine and every locally 
nilpotent derivation of $S$ is zero \cite[Theorem 3]{CPWZ2}. 
So, these numbers deserve further attention.
It would be nice to have a simple formula for $f_i$ in terms
of $\{p_{ij}\}$.

Now we define the ``ozone'' version of three invariants 
studied in \cite{KZ}. 
The {\it ozone Jacobian} of $S$ 
is defined to be 
\begin{align}
\label{E0.3.2}
\oj_S:=\prod_{i=1}^n x_i^{\mff_i-1},
\end{align}
the {\it ozone arrangement} of $S$ is defined to be 
\begin{align}
\label{E0.3.3}
\oa_S:=\prod_{\mff_i>1} x_i,
\end{align}
and the {\it ozone discriminant} of $S$ is defined to be 
\begin{align}
\label{E0.3.4}
\od_S:=\prod_{\mff_i>1} x_i^{\mff_i}=\oj_S \; \oa_S.
\end{align}
All three of these definitions depend on the chosen generating 
set $\{x_1,\dots,x_n\}$. However, by 
Proposition~\ref{zzpro1.8}(1), up to an element of 
$\kk^{\times}$, $\oj_S$, $\oa_S$, and $\od_S$ are algebra 
invariants of $\SP$.

We also consider the 
\emph{product of generators} of $S$ which is defined to be
\[ \pg_S:=\prod_{i=1}^n x_i.\]
Note that  $\pg_S$ is not an algebra invariant.

We now summarize our main results and the structure of this 
paper. In Section~\ref{zzsec1}, we recall some basic 
definitions. In Section~\ref{zzsec2}, we work out some basic 
facts concerning the reflections in $O$. In Sections~\ref{zzsec3}, 
\ref{zzsec4}, and \ref{zzsec5}, the main results are proven. 
More details on those sections are given below. We think the 
results are inspiring, though the proofs are not difficult.
At the end, in Section \ref{zzsec6}, we list some questions 
and give some examples. 

\subsection{Auslander's Theorem and smallness}
\label{zzsec0.1}
In \cite{Au}, Auslander proved that for $V$ a finite-dimensional 
$\kk$-space, $A=\kk[V]$, and $G$ a finite subgroup of $\GL(V)$, 
the map
\begin{align}\label{E0.3.5}
A\# G  &\to \End_{A^G}(A)\\  
a \# g &\mapsto 
\left(\begin{matrix}A & \to & A \\ b & \mapsto & ag(b)
\end{matrix}\right)
\notag
\end{align}
is an isomorphism if and only if $G$ is small (contains no 
pseudo-reflections). This map may be defined for any algebra 
$A$ and any finite subgroup $G$ of $\Aut(A)$, though in general 
it may not be injective or surjective. We say Auslander's 
Theorem holds for the pair $(A,G)$ (or $(A,G)$ satisfies
Auslander's Theorem) if \eqref{E0.3.5} is an isomorphism.

Auslander's Theorem is a critical component in the study 
of the McKay correspondence, and has been studied in the 
noncommutative setting \cite{CKWZ2,CYZ,GKMW, Zhu22}. Bao, 
He, and the fourth author introduced the notion of pertinency 
in \cite{BHZ2,BHZ1}. As above, let $A$ be an algebra and $G$ 
a finite subgroup of $\Aut(A)$. Then the \emph{pertinency} of 
the $G$ action on $A$ is defined to be
\begin{align}
\label{E0.3.6}
\p(A,G)=\GKdim(A)-\GKdim\left( \frac{A\#G}{(f_G)}\right)
\end{align}
where $f_G = \sum_{g \in G} 1\# g \in A\#G$.

In Section~\ref{zzsec3} we study Auslander's Theorem for the 
pair $(S,O)$. There is a notion of a reflection in the 
noncommutative setting related to the trace series of a graded 
algebra [Definition \ref{zzdef2.1}(1)]. However, for the skew 
polynomial rings $S$, a diagonal automorphism $g$ is a 
reflection if and only if it is a classical pseudo-reflection 
when restricted to $\bigoplus_{i=1}^n \kk x_i$ [Example 
\ref{zzex2.2}].

\begin{theorem}
\label{zzthm0.4}
Assume Hypothesis~\ref{zzhyp0.3}. 
The following are equivalent.
\begin{enumerate}
\item[(1)] 
Auslander's Theorem holds for $(S,O)$.
\item[(2)] 
The pertinency satisfies $\p(S,O)\geq 2$.
\item[(3)]
The $O$-action is small in the sense of 
Definition~\ref{zzdef2.1}(3).
\item[(4)] 
The $O$-action is small in the classical sense, i.e., 
it contains no pseudo-reflection when restricted to 
$\bigoplus_{i=1}^n \kk x_i$.
\item[(5)]
The only Artin--Schelter regular algebra $T$ satisfying 
$Z \subseteq T \subseteq S$ is $T = S$.
\item[(6)]
The ozone Jacobian $\oj_S=1$, namely, $\mff_i=1$ for all $i$.
\item[(7)]
For each $i$, there is an element of the form 
$x_1^{a_1}\cdots x_i \cdots x_{n}^{a_n}$ in $Z$.
\end{enumerate}
\end{theorem}

Part (6) of the above theorem shows that $\oj_S$ (or the set 
$\{\mff_1, \dots, \mff_n\}$) serves as an invariant which 
can be used to determine whether Auslander's Theorem holds 
for $(S,O)$. For small $n$, we have the following result in 
terms of the matrix $B = (b_{ij})$. 

\begin{theorem}\label{zzthm0.5}
Assume Hypothesis~\ref{zzhyp0.3}.
\begin{enumerate}
\item[(1)] 
Let $n=2$. Then $(S,O)$ does not satisfy Auslander's Theorem.
\item[(2)] 
Let $n=3$. Then $(S,O)$ satisfies Auslander's Theorem if and 
only if $\gcd(b_{ij},\ell)=1$ for each $i \neq j$.
\item[(3)] 
Let $n=4$. Then $(S,O)$ satisfies Auslander's Theorem if and 
only if 
\begin{itemize}
\item 
$\pf(B) \equiv 0 \pmod{\ell}$, and 
\item 
there does not exist an index $j$ and an integer $k$ such 
that $k b_{ij} \equiv 0 \pmod{\ell}$ for all but one $i$.
\end{itemize}
\end{enumerate}
\end{theorem}

\subsection{Regular center}\label{zzsec0.2}
In Section~\ref{zzsec4} we consider the question of when 
$Z:=\ZSP$ is regular, that is, under what conditions is $Z$ 
a polynomial ring. Again, in the case $n=2$ this is clear 
(see Theorem~\ref{zzthm0.7}(1)).
For the next theorem, we refer to Definition \ref{zzdef1.6} for the notation $\fj_{S,O}$ and $\fa_{S,O}$.

\begin{theorem}\label{zzthm0.6}
Assume Hypothesis~\ref{zzhyp0.3}. 
The following are equivalent.
\begin{enumerate}
\item[(1)] 
$Z$ is regular.
\item[(2)] 
$Z=\kk[x_1^{w_1},\cdots,x_n^{w_n}]$ for some $w_i\geq 1$.
\item[(3)] 
$Z=\kk[x_1^{\mff_1},\cdots,x_n^{\mff_n}]$.
\item[(4)] 
$S$ is a free module over $Z$.
\item[(5)] 
$O$ is a reflection group.
\item[(6)] 
$|O|=\frac{\ell^n}{\prod_{s=1}^{n} 
\gcd\{b_{1s}, b_{2s},\cdots,b_{ns},\ell\}}$.
\item[(7)]
$|O|=\prod_{i=1}^n \mff_i$.
\item[(8)] 
$O\cong \prod_{s=1}^{n} \langle \phi_s\rangle$.
\item[(9)]
$x_i^{\mff_i}\in Z$ for all $i$.
\end{enumerate}
When $Z$ is regular, then $\oj_S$ is equal to
$\fj_{S,O}$ and $\oa_S$ is equal to $\fa_{S,O}$.
\end{theorem}

For small $n$, we have the following theorem in terms 
of the parameters $\bp$ or the matrix $B$. This result 
makes use of the Smith normal form of the matrix $B$.

\begin{theorem}
\label{zzthm0.7}
Assume Hypothesis~\ref{zzhyp0.3}.
\begin{enumerate}
\item[(1)] 
Let $n=2$. Then $Z = \kk[x_1^{\ell}, x_2^{\ell}]$ so $Z$ 
is regular.
\item[(2)] 
Let $n=3$. Then $Z$ is regular if and only if the orders 
of $\{p_{ij}\}_{i<j}$ are pairwise coprime.
\item[(3)] 
Let $n=4$. Then $Z$ is regular if and only if the orders 
of $\{q_{ij}\}_{i<j}$ are pairwise coprime, where the 
$q_{ij}$ are defined in Corollary \ref{zzcor4.8}. 
\end{enumerate}
\end{theorem}

\begin{remark}
\label{zzrem0.8}
In the case where $\ell\mid \pf(B)$, we have $q_{ij}=p_{ij}$, 
so the conditions for $Z$ to be regular for $n=3$ and $4$ 
look similar. 
\end{remark}

For arbitrary $n$, we have the following interesting partial
result. Let $o(p_{ij})$ denote the order of the root of unity $p_{ij}$. 

\begin{theorem}
\label{zzthm0.9}
Assume Hypothesis~\ref{zzhyp0.3}. Suppose the orders  
$\{o(p_{ij})\}_{i<j}$ are pairwise coprime. Then $Z$ is 
regular. As a consequence,
\begin{enumerate}
\item[(1)]
$\mff_i=\prod_{j}o(p_{ij})=\prod_{j\neq i} o(p_{ij})$
for all $i$.
\item[(2)]
$|O|=\ell^2$.
\end{enumerate}
\end{theorem}

Using the above theorem, one can easily construct many examples of skew polynomial rings with regular centers. We are not aware of any analogous result for 
other families of Artin--Schelter regular algebras.

\subsection{Gorenstein center}
\label{zzsec0.3}
In Section~\ref{zzsec5} we study the question of when $Z:=\ZSP$ 
is Gorenstein. By a theorem of Watanabe \cite[Theorem 1]{W1}, 
the invariant ring of $\kk[x_1,\hdots,x_n]$ by a finite linear 
group $K$ is Gorenstein when $K \subseteq \SL_n(\kk)$. There is 
a suitable replacement for the group $\SL_n(\kk)$ in the 
noncommutative setting using the \emph{homological determinant} 
$\hdet: K \to \kk^\times$, as introduced by Jorgensen and the 
fourth author \cite{JoZ}. We refer to that reference for the 
full generalization, as computing $\hdet$ in the case of the 
elements of $\Oz(S)$ acting on $S$ is relatively trivial
[Example~\ref{zzex1.4}]. 

\begin{theorem}\label{zzthm0.10}
Assume Hypothesis~\ref{zzhyp0.3}. 
The following are equivalent.
\begin{enumerate}
\item[(1)]
$S$ is Calabi--Yau {\rm{(}}see Definition~\ref{zzdef1.2}{\rm{)}}.
\item[(2)]
The $O$-action on $S$ has trivial homological determinant.
\item[(3)]
$\prod_{s=1}^n p_{is}=1$ for all $i=1,\dots,n$.
\item[(4)]
$\pg_S\in Z$.
\end{enumerate}
As a consequence, if one of the above holds, then $Z$ is 
Gorenstein and all statements in Theorems~\ref{zzthm0.4} 
and~\ref{zzthm0.11} hold.
\end{theorem}

When the $O$-action has trivial homological determinant, then 
$Z$ is Gorenstein \cite[Theorem 3.3]{JoZ}. More generally, 
we let $H$ denote the subgroup of $O$ generated by reflections. 
Then $O/H$ acts on $S^H$. By \cite[Theorem 0.2]{KKZ4}, $Z=S^O$ 
is Gorenstein if and only if the $O/H$-action has trivial 
homological determinant. By a result of Kirkman and the fourth 
author \cite[Theorem 2.4]{KZ}, we can show that $Z$ is 
Gorenstein if and only if the ozone Jacobian $\oj_S$ is equal 
to the Jacobian $\fj_{S,O}$ defined in \cite[Definition 0.3]{KZ} 
(also see Definition~\ref{zzdef1.6}). We summarize these results 
as follows.

\begin{theorem}
\label{zzthm0.11}
Assume Hypothesis~\ref{zzhyp0.3} and let $H$ denote the subgroup 
of $O$ generated by reflections. The following are equivalent.
\begin{enumerate}
\item[(1)]
$Z$ is Gorenstein.
\item[(2)]
The $O/H$-action on $S^H$ has trivial homological determinant.
\item[(3)]
$\phi_i(\oj_S\; \pg_S)=\oj_S\; \pg_S$, or 
$\prod_{s=1}^n p_{is}^{\mff_s}=1$ for all $i$.
\item[(4)]
$\oj_S\; \pg_S\in Z$.
\end{enumerate}
\end{theorem}

By the above theorem, the centrality of $\oj_S \pg_S$ serves 
an indicator for $Z$ being Gorenstein. Similarly, by 
Theorem~\ref{zzthm0.10}, the centrality of $\pg_S$ serves as 
an indicator for $S$ being Calabi--Yau. 
Using Theorems \ref{zzthm0.4}, \ref{zzthm0.6}, \ref{zzthm0.10},
and \ref{zzthm0.11}, one can easily show the following.

\begin{corollary}
\label{zzcor0.12}
Assume Hypothesis \ref{zzhyp0.3}. Then the following hold.
\begin{enumerate}
\item[(1)]
$S$ is Calabi--Yau if and only if $Z$ is Gorenstein and 
Auslander's Theorem holds for $(S ,O)$.
\item[(2)]
If $S$ is Calabi--Yau, then $Z$ is not regular.
\end{enumerate}
\end{corollary}

It is natural to ask 
if there are similar results for other classes of 
Artin--Schelter regular algebras. For example, if $A$ is 
a noetherian graded down-up algebra with parameters
$(\alpha,\beta)$, then there is a canonical element
$\fb\in A$ such that $A$ is Calabi--Yau if and only if 
$\fb\in Z(A)$ [Proposition~\ref{zzpro6.4}]. It would be 
interesting to know if $\fb$ can be defined homologically.

For small $n$, we are able to give conditions which are
equivalent to the Gorensteinness of $Z$ 
in terms of parameters ${\mathbf p}$ or the matrix $B$.

\begin{theorem}\label{zzthm0.12}
Assume Hypothesis~\ref{zzhyp0.3}. 
\begin{enumerate}
\item[(1)]
Let $n=2$. Then $Z=\kk[x_1^\ell,x_2^\ell]$ so $Z$ is 
Gorenstein.
\item[(2)] 
Let $n=3$. For each $1 \leq i, j \leq 3$, let 
$b_{ij}' = \gcd(b_{ij}, \ell)$. Then $Z$ is Gorenstein 
if and only if 
\[ \overline{B}(b_{23}',b_{13}',b_{12}')^T = 0\]
in $\overline{\ZZ}^3$.
\item[(3)] 
Let $n=4$. Then $Z$ is Gorenstein if and only if 
\[ \frac{\ell}{\gcd(\pf(B),\ell)} 
\overline{B} (v_1,v_2,v_3,v_4)^T = 0\]
in $\overline{\ZZ}^4$ where 
$v_i = \gcd(\ell, \{b_{jk} \mid j,k \neq i\})$ for 
$i=1,\hdots,4$.
\end{enumerate}
\end{theorem}

\subsection{Graded isolated singularities}\label{zzsec0.4}
By Remark \ref{zzrem3.7}, under Hypothesis~\ref{zzhyp0.3}, 
$Z$ is regular or does not have isolated singularities. 

\section{Preliminaries}
\label{zzsec1}

In this section, we review some concepts that will be used
throughout the paper. Some of the definitions can be found 
in the survey paper of Kirkman \cite{Ki}. 

\begin{definition}
\label{zzdef1.1}
A connected graded algebra $A$ is called {\it Artin--Schelter 
Gorenstein} (or {\it AS Gorenstein}, for short) if $A$ has 
injective dimension $d<\infty$ on the left and on the right, 
and
\[\Ext^i_A({}_A\kk, {}_{A}A)\cong \Ext^i_{A}(\kk_A,A_A)\cong 
\delta_{id} \kk(\bfl)\]
where $\delta_{id}$ is the Kronecker-delta. If in addition $A$ 
has finite global dimension and finite Gelfand--Krillov 
dimension ($\GKdim$), then $A$ is called {\it Artin--Schelter regular} 
(or {\it AS regular}, for short) of dimension $d$.
\end{definition}

For an algebra $A$, the \emph{enveloping algebra} of $A$ is 
$A^e=A \otimes A^{\mathrm{op}}$. If $\sigma$ is an automorphism 
of $A$, then $^{\sigma} \! A$ is the $A^e$-module which, 
as a $\kk$-vector space is just $A$, but where the natural action 
is twisted by $\sigma$ on the left: that is,
\[ (a \otimes b)\cdot c = \sigma(a)cb\]
for all $a \otimes b \in A^e$ and $c \in A$.

\begin{definition}\label{zzdef1.2}
Suppose $A$ is AS regular of dimension $d$. Then there is a graded 
algebra automorphism $\mu$ of $A$, called the 
\emph{Nakayama automorphism}, such that
\[ \Ext_{A^e}^d(A,A^e) \cong {}^{\mu} \! A.\]
When $\mu = \Id$, then we say that $A$ is \emph{Calabi--Yau}.
\end{definition}

Next we recall the trace series and homological determinant.

\begin{definition}\label{zzdef1.3}
Let $A$ be a connected graded algebra with $A_j$ denoting the
the homogeneous part of $A$ of degree $j$. The {\it trace series}
of a graded algebra automorphism $\sigma$ of $A$ is defined to 
be
\[\Tr_A(\sigma,t):=\sum_{j=0}^{\infty} (\tr(\sigma\restrict{A_j})) t^j,\]
where $\tr(\sigma\restrict{A_j})$ is the usual trace of the $\kk$-linear 
map $\sigma$ restricted to $A_j$.
\end{definition}

When $\sigma$ is the identity, we recover the Hilbert series 
of $A$, namely, $\Tr_A(\Id, t)=h_A(t)$. If $A$ is AS regular, 
then $h_A(t)=1/r(t)$ where $r(t)$ is an integral polynomial 
of degree $\bfl\geq d:=\gldim A$ and 
$r(t)=1+r_1 t+\cdots +r_{\bfl-1} t^{\bfl-1}+r_{\bfl} t^{\bfl}$
where $r_{\bfl}=(-1)^d$. For every graded algebra 
automorphism $\sigma$ of $A$, $\Tr_A(\sigma,t)$ must be equal 
to $1/q(t)$ where 
$q(t)=1+q_1 t+\cdots +q_{\bfl-1} t^{\bfl-1}+q_{\bfl} t^{\bfl}$.
By \cite[Lemma 2.6]{JoZ}, the homological determinant of $\sigma$
is defined to be
\begin{align}\label{E1.3.1}
\hdet(\sigma)=(-1)^d q_{\bfl}=r_{\bfl} q_{\bfl}.
\end{align}
One nice property of homological determinant is that the map
\[\hdet: \Autgr(A)\to \kk^{\times}\]
is a group homomorphism \cite[Proposition 2.5]{JoZ}. When we 
consider a finite group $G$ acting on $A$, we say that the 
$G$-action has trivial homological determinant if $\hdet(G)=1$. 

\begin{example}\label{zzex1.4} 
Let $A$ be the weighted skew polynomial ring $\SP$ with 
$\deg(x_i)>0$ for all $i$. Let $\sigma$ be a diagonal automorphism 
of $A$ determined by
\[\sigma(x_i)=a_i x_i\]
for all $i$, where $a_i\in \kk^{\times}$. Such an automorphism
is also denoted by $\diag(a_i)$. It is easy to see that 
the trace series of $\sigma$ is 
\[\Tr_A(\sigma,t)= \prod_{i=1}^n(1-a_it^{\deg(x_i)})^{-1}.\]
By \eqref{E1.3.1}, the homological determinant 
of $\sigma$ is 
\begin{align}\label{E1.4.1}
	\hdet(\sigma)=\prod_{i=1}^n a_i.
\end{align} 
Another way of expressing the $\hdet$ is
\begin{align}\label{E1.4.2}
	\sigma\left(\prod_{i=1}^n x_i\right)=\hdet(\sigma)\left(\prod_{i=1}^n x_i\right).
\end{align}
Note that in this example, $\hdet(\sigma)$ is the determinant 
of the matrix $\diag(a_i)$. In general this is not true, 
see \cite[Example 1.8(2)]{Ki}.
\end{example}

\begin{lemma}\label{zzlem1.5}
Assume Hypothesis~\ref{zzhyp0.3}.
\begin{enumerate}
\item[(1)]
For all $i$, $\hdet \phi_i=\prod_{s=1}^n p_{is}$.
\item[(2)]
$O$ has trivial homological determinant if and 
only if $\prod_{s=1}^n p_{is}=1$ for all $i$.
\item[(3)]
If $n=3$, then $O$ has trivial homological determinant
if and only if there exists a root of unity $p$ such that
\[
(p_{ij})=\begin{pmatrix} 1& p & p\inv\\
p\inv & 1 & p\\ 
p & p\inv & 1  \end{pmatrix}.
\]
\item[(4)]
If $n=4$, then $O$ has trivial homological determinant if 
and only if there exist roots of unity $p,q,r$ such that
\[(p_{ij})=\begin{pmatrix} 1& q & p &p^{-1}q^{-1}\\
q^{-1} &1& r &q r^{-1}\\
p^{-1} & r^{-1} & 1 & pr\\
pq& q^{-1} r& p^{-1}r^{-1}&1
\end{pmatrix}.\]
\end{enumerate}
\end{lemma}
\begin{proof}
Part (1) follows directly from \eqref{E1.4.1}.
Parts (2), (3), and (4) now follow from (1).
\end{proof}

The notions of the Jacobian and reflection arrangement of a 
finite group action on an AS regular algebra were introduced 
in \cite[Definition 0.3]{KZ}. We will soon consider diagonal 
actions on $\SP$ after the definition (for example, $G$ is a 
subgroup of $O$).

\begin{definition}[{\cite[Definition 0.3]{KZ}}]\label{zzdef1.6}
Let $A$ be a noetherian AS regular algebra and $G$ is a
finite subgroup of $\Autgr(A)$. 
\begin{enumerate}
\item[(1)]
Let $A_{\hdet^{-1}}:
=\{ x\in A\mid \sigma(x)= (\hdet(\sigma))^{-1} x,\;
\forall \; \sigma\in G\}$. If $A_{\hdet^{-1}}$ is a free 
$A^G$-module of rank one on both sides generated by an element,
denoted by $\fj_{A,G}$, then $\fj_{A,G}$ is called the 
{\it Jacobian} of the $G$-action. By \cite[Theorem 2.4]{KZ}, 
the Jacobian $\fj_{A,G}$ exists if and only if $A^G$ is AS 
Gorenstein. 
\item[(2)]
Let $A_{\hdet}:=\{ x\in A\mid \sigma(x)= \hdet(\sigma)  
x,\; \forall \; \sigma\in G\}$. If $A_{\hdet}$ is a free 
$A^G$-module of rank one on both sides generated by an 
element, denoted by $\fa_{A,G}$, then $\fa_{A,G}$ is 
called the {\it reflection arrangement} of the $G$-action.
By \cite[Theorem 0.2]{KZ}, the reflection arrangement 
$\fa_{A,G}$ exists if $A^G$ is AS regular. 
\end{enumerate}
\end{definition}

\begin{lemma}\label{zzlem1.7}
Assume Hypothesis~\ref{zzhyp0.3} and let $G$ be a finite 
subgroup of diagonal automorphisms of $S$. Then
$\pg_S:=\prod_{i=1}^n x_i$ is an element in 
$S_{\hdet}$. As a consequence, if $\fa_{S,G}$ exists, 
then it is a factor of $\pg_S$. 
\end{lemma}

\begin{proof} 
The main assertion follows from \eqref{E1.4.2}.
The consequence is clear.
\end{proof}

Recall that $\oj_S$, $\oa_S$, and $\od_S$ are 
defined \eqref{E0.3.2}--\eqref{E0.3.4}. 

\begin{proposition}\label{zzpro1.8}
Assume Hypothesis~\ref{zzhyp0.3}. 
\begin{enumerate}
\item[(1)]
Up to nonzero scalars, $\oj_S$, $\oa_S$, and $\od_S$ are 
algebra invariants of $S$. Namely, they are independent 
of the chosen generating set $\{x_1,\dots,x_n\}$.
\item[(2)]
The set $\{\mff_1,\dots, \mff_n\}$ with multiplicities 
is an algebra invariant of $S$.
\end{enumerate}
\end{proposition}

\begin{proof} (1) We use the notation introduced in 
\cite{CPWZ2} without giving detailed definitions. For each 
$s \in \{1,\hdots,n\}$, let $T_s$ be the set defined as 
before \cite[Lemma 2.9]{CPWZ2} and let $d_w(S/Z)$ be the 
discriminant of $S$ over its center $Z$ introduced in 
\cite[Definition 1.2(3)]{CPWZ2} where $w=\rk(S_Z)$. It 
follows from \cite[Lemma 2.9(2)]{CPWZ2} that $\mff_s>1$ 
if and only if $T_s$ is empty. Combined with 
\cite[Theorem 2.11(2)]{CPWZ2}, we obtain that 
$\mff_s>1$ if and only if $x_s$ divides $d_w(S/Z)$,
In other words, $d_w(S/Z)=_{\kk^{\times}}\prod_{i, \mff_i>1} 
x_i^{\lambda_i}$ for some $\lambda_i\geq 1$. By definition,
$d_w(S/Z)$ is an algebra invariant of $A$ up to a nonzero 
scalar \cite{CPWZ2}.

Since $S$ is ${\mathbb N}^n$-graded, the prime decomposition 
of $x_i^{\lambda_i}$ is unique and $\{x_i\mid \mff_i >1\}$ is 
a complete list of prime factors of $d_w(S/Z)$. Therefore the 
list $I:=\{\kk^{\times} x_i\mid \mff_i>1\}$ (after adding 
nonzero scalars) is an algebra invariant of $S$. As a 
consequence, $\oa_S$ is an algebra invariant.

Note that, for each $\kk^{\times} x_i\in I$, 
\[\mff_i=\gcd\{ d\mid x_i^d a\in Z  {\text{ for some $a\in S$ with $x_i\nmid a$}}\}.\]
This implies that the set $\{\kk^{\times}x_i^{\mff_i} \mid 
\mff_i>1\}$ is also an algebra invariant of $S$. 
As a consequence, $\od_S$ is an algebra invariant of $S$.

Since $\oj_S=\od_S(\oa_S)^{-1}$, $\oj_S$ is an 
algebra invariant.

(2) By part (1), the set $\{\mff_i\mid \mff_i>1\}$ is an 
algebra invariant of $S$. Since $\{\mff_1,\dots,\mff_n\}
=\{\mff_i\mid \mff_i>1\}\cup \{1,\dots,1\}$ considered as
a set with multiplicities, the assertion follows.
\end{proof}

\section{Reflections and Reflection groups}\label{zzsec2}

Throughout this section we assume Hypothesis~\ref{zzhyp0.3}.
Recall that a linear automorphism $g$ of a finite-dimensional 
vector space $V$ is a \emph{pseudo-reflection} if $g$ fixes a 
codimension one subspace. In the noncommutative case, we need 
to use the trace series [Definition~\ref{zzdef1.3}] to define 
a reflection.

\begin{definition}[{\cite[Definition 1.4]{KKZ5}}]
\label{zzdef2.1}
Let $A$ be an AS regular algebra of Gelfand--Kirillov 
dimension $n$. 
\begin{enumerate}
\item[(1)]
We call a graded algebra automorphism 
$\sigma$ of $A$ a {\it reflection} of $A$ if $\sigma
\neq \Id$ has finite order and the trace series 
of $\sigma$ has the form
\[\Tr_A(\sigma,t)=\frac{1}{(1-t)^{n-1}q(t)} \quad
{\text{where}} \;\; q(1)\neq 1.\]
\item[(2)]
A finite subgroup $G\subseteq \Autgr(A)$ is called 
a {\it reflection group} if it is generated by reflections.
\item[(3)]
A finite subgroup $G\subseteq \Autgr(A)$ is called 
{\it small} if it does not contain any reflections. 
\end{enumerate}
\end{definition}

\begin{example}
\label{zzex2.2} 
Let $S=\SP$ with $\deg(x_i)=1$ for all $i$ and let 
$\sigma=\diag(a_i)$, $a_i \in \kk^\times$, as in 
Example~\ref{zzex1.4}. By that example the trace series of 
$\sigma$ is 
\[\Tr_S(\sigma,t) = \prod_{i=1}^n(1-a_it)\inv.\]
Hence $\sigma$ is a reflection if and only if $a_i=1$
for all $i$ but one. Therefore $\sigma$ is a reflection
if and only if $\sigma\restrict{S_1}$ is a pseudo-reflection. 
As a consequence, $O$ is small in the sense of 
Definition~\ref{zzdef2.1}(3) if and only if $O$ 
is small in the classical sense.
\end{example}

One noncommutative version of the Shephard--Todd--Chevalley
theorem is the following.

\begin{theorem}[{\cite[Theorem 5.5]{KKZ3}}]\label{zzthm2.3}
Let $S=\SP$ and let $G$ be a finite group of graded 
algebra automorphisms of $S$. Then $S^G$ has finite 
global dimension if and only if $G$ is generated by 
reflections of $S$ {\rm{(}}in this case $S^G$ is again
a skew polynomial ring with weighted grading{\rm{)}}.
\end{theorem}

In this paper we will only apply this theorem when every 
element of $G$ is a diagonal automorphism of $S$. In fact, 
we will only consider the case when $G$ is a subgroup of 
$\Oz(S)$. We now introduce the notion of an ozone subring.

\begin{definition}\label{zzdef2.4}
Let $A$ be a noetherian PI AS regular algebra with center $Z$.
\begin{enumerate}
\item[(1)]
A subring $R$ of $A$ is called {\it ozone} if $R$ is AS 
regular and $Z\subseteq R \subseteq A$.
\item[(2)]
The set of all ozone subrings of $A$ is denoted by $\Phi_Z(A)$.
\item[(3)]
If $R$ is a minimal element in $\Phi_Z(A)$ via inclusion, then 
$R$ is called a {\it mozone} subring of $A$.
\end{enumerate}
\end{definition}

Note that $A$ itself is an ozone subring of $A$, so $\Phi_Z(A)$ 
is not empty. If $Z$ is AS regular (i.e., a commutative 
polynomial ring) or if $A = Z$ is itself commutative, then $Z$ 
is the unique mozone subring of $A$. Assume 
Hypothesis~\ref{zzhyp0.3} and let $H$ be the subgroup of 
$O:=\Oz(S)$ that is generated by reflections. Let $B := S^H$. 
By Theorem~\ref{zzthm2.3}, $B$ is a skew polynomial ring (and 
hence an ozone subring of $S$). We will show that $B$ is a 
mozone subring of $S$. (See also Question~\ref{zzque6.3}(6).)

\begin{proposition}\label{zzpro2.5}
Assume Hypothesis~\ref{zzhyp0.3}. Let $H$ be the subgroup of 
$O$ generated by reflections in $O$. Then $S^H$ is a mozone 
subring of $S$.
\end{proposition}
\begin{proof}
Since $H$ fixes $S^H$, there is an induced action of $O/H$ on 
$S^H$. By the proof of \cite[Proposition 4.12]{KKZ4}, this 
action is small in the sense of Definition~\ref{zzdef2.1}(3). 
Hence, by \cite[Theorem 5.5]{BHZ2}, we therefore have that 
$S^H \# (O/H) \cong \End_{S^O}(S^H)$ (which is analogous to 
a weighted version of Theorem~\ref{zzthm0.4}).

Now suppose for contradiction that there is an AS regular 
algebra $R$ such that $Z \subseteq R \subsetneq S^H$. By 
\cite[Lemma 1.10(b)]{KKZ5}, $S^H$ is a graded free module 
over $R$. Hence, we can write 
$S^H = R \oplus R(-d_1) \oplus \cdots \oplus R(-d_n)$ for 
some integers $d_i$, at least one of which is positive. Then 
$\End_{S^O}(S^H)$ contains elements of negative degree, which 
contradicts that it is isomorphic to $S^H \# (O/H)$.
\end{proof}

Recall from \eqref{E0.3.1} that 
\[ \mff_i = \gcd\{d_i \mid \exists~d_1,\dots,\widehat{d_i},\dots,d_n ~\text{with}~ x_1^{d_1}\cdots x_i^{d_i}\cdots x_n^{d_n}\in Z\}.\]
Let $M$ be the subring of $S$ generated by 
$\{x_1^{\mff_1},\dots,x_n^{\mff_n}\}$. 
It is straightforward to verify that $M\in \Phi_Z(S)$. 

\begin{proposition}\label{zzpro2.6}
Assume Hypothesis~\ref{zzhyp0.3}.
Let $H$ be the subgroup of $O$ generated by reflections in $O$.
\begin{enumerate}
\item[(1)]
The subgroup $H$ is isomorphic to $\prod_{i=1}^n \langle r_i
\rangle$ where $r_i$ is a diagonal automorphism of $S$ 
determined by 
\[ r_i:x_j\mapsto \begin{cases}
x_j & j\neq i\\
c_i x_i & j=i\end{cases}\] 
for some root of unity $c_i$.
\item [(2)]
For each $i$, let $w_i$ be the order of $c_i$ given in 
part {\rm{(1)}}. For each $i,j$, define 
$q_{ij} = p_{ij}^{w_i w_j}$ and let $\bq = (q_{ij})$.
Then $S^H = \kk_{\bq}[x_1^{w_1},\ldots,x_n^{w_n}]$.
\item[(3)] 
For each $i$, $w_i=\mff_i$. 
\item[(4)]
$\mff_i = \min\{d_i>0 \mid \exists~d_1,\dots,\widehat{d_i},\dots,d_n ~\text{with}~ x_1^{d_1}\cdots x_i^{d_i}\cdots x_n^{d_n}\in Z\}$.
\end{enumerate}
\end{proposition}

\begin{proof}
(1) Let $H_i$ be the subgroup of $H$ that is generated by 
reflections $r$ of the form 
\[ r(x_j)=\begin{cases} 
	x_j & j\neq i\\
	\alpha_i x_i & j=i
\end{cases}\] 
for some $\alpha_i\in \kk$.
Since every element in $O$ is a diagonal automorphism, 
by Example~\ref{zzex2.2}, every reflection is of the form 
of $r$ given above. Since $H$ is generated by reflections 
in $O$, $H$ is generated by the union of the $H_i$. It is clear
that the product $\prod_{i=1}^n H_i$ is a subgroup of $H$.
Thus $H=\prod_{i=1}^n H_i$. Note that each $H_i$ is a
finite subgroup of $\kk^{\times}$, whence it is cyclic.
Therefore each $H_i$ is of the form $\langle r_i \rangle$ 
where $r_i$ is a diagonal reflection automorphism of $S$.
The assertion follows.

(2) It is clear that $S^H\supseteq 
\kk_{\bq}[x_1^{w_1},\dots,x_n^{w_n}]$. Since $S$ 
is $\ZZ^n$-graded and $H$ consists of 
$\ZZ^n$-graded algebra automorphisms of $S$, 
$S^H$ is $\ZZ^n$-graded. So to prove $S^H\subseteq 
\kk_{\bq}[x_1^{w_1},\dots,x_n^{w_n}]$, 
we only need to consider monomials in $S^H$. Suppose that 
$f:=x_1^{v_1}\cdots x_n^{v_n}$ is an element of $S^H$. Since 
$f=r_i(f)=c_i^{v_i}f$ for all $i$, we see that $w_i\mid v_i$.
The assertion follows. 

(3) By part (2) and the fact that $H\subseteq O$, we have
\[Z=S^O\subseteq S^H=\kk_{\bq}[x_1^{w_1},\dots, x_n^{w_n}].\]
Therefore every monomial in $Z$ is a product of $x_1^{w_1}, 
\dots, x_n^{w_n}$. By definition, $w_i\mid \mff_i$ for all $i$. 
Next we show that $\mff_i\mid w_i$. Let $c'_i$ be a primitive 
$\mff_i$th root of unity. By the definition of the $\mff_i$, we 
obtain that $Z$ is a subring of $M':=\kk_{(q'_{ij})}[x_1^{\mff_1},
\dots,x_{n}^{\mff_n}]$  where $q'_{ij}= p_{ij}^{\mff_i\mff_j}$. 
Let $r'_i$ be the diagonal automorphism of $S$ determined by 
\[ x_j\mapsto 
\begin{cases} 
	x_j & j\neq i\\
	c'_i x_i & j=i.
\end{cases}\] 
It is clear that $r'_i$ 
preserves $M'$ and hence $Z$ so $r_i' \in O$.
Since $r_i'$ is a reflection, therefore $r'_i\in H$
and consequently,
$x_i^{w_i}=r'_i(x_i^{w_i})=(c'_i)^{w_i}x_i^{w_i}$. This implies
that $(c'_i)^{w_i}=1$ or $\mff_i\mid w_i$ as required.

(4) By definition, $\mff_i$ is given by \eqref{E0.3.1}.
For each $i$, there are finitely many central elements, say 
$z(1), \dots, z(s)$, defined by
\[ z(j):=x_1^{d^j_1}\cdots x_i^{d^j_i}\cdots x_n^{d^j_n}\in Z\]
such that $\mff_i=\gcd\{d^{1}_i,\cdots,d^{s}_i\}$.
Write $\mff_i=\sum_{j=1}^s c_j d^j_i$. Choose an integer $\alpha$ 
such that $c'_j:=c_j+\alpha \ell>0$ for all $j$. Then 
$\mff_i= \left(\sum_{j =1}^{s} c'_j d^j_i\right) - \beta \ell$ 
for some $\beta$ and, up to a scalar, we have
$$\prod_{j} z(j)^{c'_j}=
x_1^{d'_1}\cdots x_i^{\mff_i}\cdots x_n^{d'_n} x_i^{\beta\ell}
\in Z.$$
This implies that 
$x_1^{d'_1}\cdots x_i^{\mff_i}\cdots x_n^{d'_n}\in Z$.
The claim follows.
\end{proof}

\begin{theorem}\label{zzthm2.7}
Assume Hypothesis~\ref{zzhyp0.3}.
If $Z$ is Gorenstein, then $\oj_{S}$
is equal to the Jacobian $\fj_{S,O}$.
\end{theorem}
\begin{proof} 
Let $H$ be the subgroup of $O$ generated by reflections in 
$O$. By Proposition~\ref{zzpro2.6}(2,3),
$S^H=\kk[x_1^{\mff_1},\cdots, x_n^{\mff_n}]$. Then 
\begin{align}\label{E2.7.1}
S=\bigoplus_{0\leq s_i\leq \mff_i-1}
x_1^{s_1}\cdots x_n^{s_n} S^H.
\end{align} 
Since $\pg_S\in S_{\hdet_O}$ (see \eqref{E1.4.2}) and 
since the $H$-action on $\chi:=\prod_{i=1}^n x_i^{\mff_i}$ 
is trivial, $\oj_S=\chi (\pg_S)^{-1} \in S_{\hdet_H^{-1}}$. 
The decomposition \eqref{E2.7.1} shows that 
$S_{\hdet_H^{-1}}=\oj_S S^H$
or equivalently, $\oj_S$ is $\fj_{S,H}$. 
Therefore $S_{\hdet_O^{-1}}\subseteq S_{\hdet_H^{-1}}={\oj_S} S^H$. 

We claim that $\oj_S\in S_{\hdet_O^{-1}}$. To see that we note that
$\oj_S=\chi (\pg_S)^{-1}$. For every $g\in O$, let $\overline{g}$
be the induced algebra automorphism of $S^H$. 
Since $Z=(S^H)^{O/H}$ is Gorenstein, then the $O/H$-action
on $S^H$ has trivial homological determinant by \cite[Theorem 0.2]{KKZ4}.
Since $S^H$ is a skew
polynomial ring and $\overline{g}$ is a diagonal action, 
$O/H$-action having trivial homological determinant implies that 
\[g(\chi)=
\overline{g}\left(\prod_{i=1}^n x_i^{\mff_i}\right)
=\prod_{i=1}^n x_i^{\mff_i}=\chi,\]
see \eqref{E1.4.2}. Hence
\[g(\oj_S)=g(\chi \pg_S^{-1})=g(\chi)(g(\pg_S))^{-1}
=\chi \hdet^{-1}(g) (\pg_S)^{-1}=\hdet^{-1}(g)\oj_S.\]
This implies that $\oj_S\in S_{\hdet_O^{-1}}$. Then 
it is easy to check that $S_{\hdet_O^{-1}}=\oj_S S^O$.
In other words, $\oj_S=\fj_{S,O}$. 
\end{proof}

Proposition \ref{zzpro2.6}, and part (4) in particular, make computing the $\mff_i$ straightforward in small dimension, as demonstrated by the following result.

\begin{lemma}\label{lem.fhyper}
Assume Hypothesis~\ref{zzhyp0.3} and let $n=3$. Then
\[ \mff_1 = \gcd(b_{23},\ell), \quad \mff_2 = \gcd(b_{13},\ell), \quad \mff_3 = \gcd(b_{12},\ell).\]
\end{lemma}
\begin{proof}
We prove the first equality, as the others are similar. First note that $x_1^{b_{23}}x_2^{\ell-b_{13}}x_3^{b_{12}} \in Z$, and so $\mff_1 \mid b_{23}$.

Since $x^\ell \in Z$, then $\mff_1 \mid \ell$, so $\mff_1 \mid \gcd(b_{23},\ell)$. On the other hand, there exists $a=x_1^{u_1}x_2^{u_2}x_3^{u_3} \in Z$ with $u_1=\mff_1$ by Proposition \ref{zzpro2.6}(4). 
By considering $[x_2,a]=[x_3,a]=0$ we have
\begin{align*}
	b_{23}u_3 &\equiv b_{12}u_1 \mod \ell \\
	b_{23}u_2 &\equiv -b_{13}u_1 \mod \ell.
\end{align*}
Hence, if $d\mid b_{23}$ and $d\mid \ell$, then $d \mid b_{12}u_1$ and $d \mid b_{13}u_1$. But since $d \mid b_{23}u_1$ and $\gcd(b_{23},b_{12},b_{13})=1$, then $d \mid u_1$. Thus, $d \mid \mff_1$, so $\gcd(b_{23},\ell)\mid \mff_1$.
\end{proof}

For $i=1,2,3$, set
\[ g_i = o(\phi_i) = \frac{\ell}{\gcd(b_{ij},b_{ik},\ell)} \text{ with $\{i,j,k\} = \{1,2,3\}$}.\]
Then it follows from Lemma \ref{lem.fhyper} that
\[ g_i = \frac{\ell}{\gcd(\mff_j,\mff_k,\ell)}\text{ with $\{i,j,k\} = \{1,2,3\}$}.\]
Write $g=(g_1,g_2,g_3)$ and $\mff=(\mff_1,\mff_2,\mff_3)$. By Theorem~\ref{zzthm0.6}, 
$Z$ is regular if and only if $g=\mff$.

\section{Auslander's Theorem}\label{zzsec3}

Throughout this section we assume Hypothesis~\ref{zzhyp0.3}.
In this section, we seek to understand when the Auslander map 
\[ S \# O \to \End_{S^O}(S)\] 
is an isomorphism. Recall, from Definition~\ref{zzdef2.1}, 
that a finite subgroup $G \subseteq \Autgr(A)$ is called small 
if it does not contain any reflections. By Example~\ref{zzex2.2}, 
when $A=\SP$, $G$ is small if and only if $G$, considered as a 
subgroup of $\GL(\bigoplus_{i=1}^n \kk x_i)$, is small in the 
classical sense.

\begin{proof}[Proof of Theorem~\ref{zzthm0.4}]
(1) $\Leftrightarrow$ (2):
By \cite[Theorem 0.3]{BHZ1}, the Auslander map is an 
isomorphism for the pair $(S,O)$ if and only if 
$\p(S, O) \geq 2$. 

(1) $\Leftrightarrow$ (3) $\Leftrightarrow$ (4):
Since each $\phi_i$ is a diagonal automorphism, by 
\cite[Theorem 5.5]{BHZ2}, the Auslander map is an isomorphism 
for the pair $(S,O)$ if and only if $O$ is small in the 
classical sense, that is, if and only if $O$, when restricted 
to $\bigoplus_{i=1}^n \kk x_i$, contains no pseudo-reflections 
of $\bigoplus_{i=1}^n \kk x_i$.

(3) $\Leftrightarrow$ (5): If $O$ is small, then, by 
Proposition~\ref{zzpro2.5}, we have that $S$ is the unique 
mozone subring of $S$. Conversely, if $O$ is not small, then 
let $H \leq O$ be the non-trivial subgroup of reflections. By 
Proposition~\ref{zzpro2.6}, $S^H \subsetneq S$ is a mozone 
subring of $S$, and hence $S$ is not a mozone subring.

(3) $\Leftrightarrow$ (6): 
By Proposition~\ref{zzpro2.6}(2,3),
$S^H=\kk[x_1^{\mff_1},\cdots, x_n^{\mff_n}]$. 
Hence $H=\{1\}$ if and only if $\mff_i=1$ for all $i$. 

(6) $\Leftrightarrow$ (7): 
This follows from Proposition \ref{zzpro2.6}(4).
\end{proof}

For the rest of the section we consider the cases of small
$n$. In the $n=2$ case, with $S=\kk_p[x_1,x_2]$ for $p \neq 1$, 
we have $S^O = \kk[x_1^\ell,x_2^\ell]$ where $\ell$ is the 
order of the root of unity $p$. Consequently, $S$ is free over 
$S^O$. It follows that $\End_{S^O}(S)$ has negative degree maps 
and so the Auslander map $S\# O \to \End_{S^O}(S)$ is not an 
isomorphism. Next we consider $n=3$ and $4$.

Let $g = \phi_1^{u_1}\cdots\phi_n^{u_n}\in O$ and write 
$\bu = (u_i)$ for the (column) vector with components $u_i$.
For a vector $\bx$, we use $\bx^T$ to denote its transpose.
The restriction of $g$ to $S_1$ is given by the diagonal matrix 
$\diag(\xi^{v_i})$ where $(v_i)^T = \bv^T = \bu^T \overline{B}$, 
viewing $\bv$ and $\bu$ as elements of $\overline{\ZZ}^n$. It is 
clear that $g$ is a pseudo-reflection if $\bv = \lambda \be_i$ 
for some nonzero $\lambda \in \overline{\ZZ}$ and some 
$1 \leq i \leq n$.  We can therefore express the condition 
of $O$ being small in terms of nonexistence of solutions of 
some linear equations over $\overline{\ZZ}$. Namely, $O$ 
having no pseudo-reflections is equivalent to the equation 
$\bu^T \overline{B}  = \lambda \be_i^T$ having no 
solution for any nonzero $\lambda \in \overline{\ZZ}$. 
Taking transposes, this is equivalent to the equation
\begin{align}\label{E3.0.1}
\overline{B} \by  = \lambda \be_i
\end{align}
having no solutions $\by \in \overline{\ZZ}^n$ for any 
nonzero $\lambda \in \overline{\ZZ}$ and any $1 \leq i \leq n$.

The next lemma allows us to reduce to the case that no $x_i$ 
is central.

\begin{lemma}\label{zzlem3.1}
Suppose $x_1$ is central {\rm{(}}equivalently, $\phi_1=\id${\rm{)}}. 
Let $R$ be the subalgebra of $S$ generated by $x_2,\hdots,x_n$ 
and let $O' = \langle \left.\phi_i\right|_R \mid i=2,\hdots,n\rangle$. 
Then $\p(R,O') = \p(S,O)$ so $(S,O)$ satisfies Auslander's Theorem
if and only if $(R, O')$ does.
\end{lemma}
\begin{proof}
Since $O$ acts trivially on $x_1$, by \eqref{E0.3.6}, we have 
\begin{align*}
\p(S,O) &= \GKdim(S) - \GKdim(S\#O / (f_{O}) ) \\
    &= \GKdim(R[x_1]) - \GKdim(R[x_1]\#O / (f_{O}) ) \\
    &= (\GKdim(R) + 1) - \GKdim((R\#O' / (f_{O'}) )[x_1]) \\
    &= (\GKdim(R) + 1) - (\GKdim(R\#O' / (f_{O'}) ) + 1) \\
    &= \GKdim(R) - \GKdim(R\#O' / (f_{O'}) ) \\
    &= \p(R,O').\qedhere
\end{align*}
\end{proof}

We begin by considering the case $n=3$. For a root of unity
$p$, let $o(p)$ denote its order. 

\begin{proposition}\label{zzpro3.2}
Assume Hypothesis~\ref{zzhyp0.3} with $n=3$. Then $(S, O)$ 
satisfies Auslander's Theorem if and only if each $p_{ij}$ is 
a primitive $\ell$th root of unity for all $i \neq j$.
\end{proposition}
\begin{proof}
By Lemma~\ref{zzlem3.1}, we may assume that no $\phi_i=\Id$.

Assume $O$ is small so that no power of the $\phi_i$ are reflections. 
Note that $\phi_1^{o(p_{12})} = \diag(1,1,p_{13}^{o(p_{12})})$. 
By hypothesis, this implies that $o(p_{13})$ divides $o(p_{12})$. 
Using the same logic on the other $\phi_i$ gives 
$o(p_{12}) = o(p_{13}) = o(p_{23})$. That is, each of the $p_{ij}$ is a 
primitive $\ell$th root of unity.

Conversely, assume that each of the $p_{ij}$ is a primitive 
$\ell$th root of unity and so $\gcd(b_{ij}, \ell) = 1$ for 
each $b_{ij}$. We wish to show that equation \eqref{E3.0.1} has 
no solutions. Now we compute
\[ \overline{B} \by = \overline{B} \begin{bmatrix}
y_1 \\ y_2 \\ y_3
\end{bmatrix}  = \begin{bmatrix}
y_2 b_{12} + y_3 b_{13} \\
-y_1 b_{12} + y_3 b_{23} \\
-y_1b_{13} - y_2b_{23}
\end{bmatrix}.\]
Without loss of generality, suppose that the first two 
entries of $\overline{B} \by$ are equal to $0$. Then
\begin{align*}
b_{12}(y_1 b_{13} + y_2 b_{23}) 
&= b_{13}(y_1 b_{12}) + b_{23}(y_2 b_{12}) \\
&= b_{13}(y_3 b_{23}) - b_{23}(y_3 b_{13}) = 0
\end{align*}
and hence $\overline{B} \by = 0$. Therefore, equation 
\eqref{E3.0.1} has no solutions and so $O$ is small.
\end{proof}

\begin{example}\label{zzex3.3}
As a consequence of the proof of Proposition~\ref{zzpro3.2}, 
when $n = 3$, if no power of any $\phi_i$ is a reflection, 
then $O$ is small. By contrast, in the $n = 4$ case, it is 
possible for $O$ to contain a reflection even if no power of 
any $\phi_i$ is a reflection.

Let $\xi$ be a primitive sixth root of unity and set
\begin{align*}
p_{12} &= p_{13} = p_{23} =\xi \\
p_{14} &= p_{24} = p_{34} = -1.
\end{align*}
Then
\begin{align*}
\phi_1 &= \diag(1,\xi,\xi,-1)  &
\phi_2 &= \diag(\xi\inv,1,\xi,-1) \\
\phi_3 &= \diag(\xi\inv,\xi\inv,1,-1) &
\phi_4 &= \diag(-1,-1,-1,1).
\end{align*}
(Here we consider the map $\phi_i$ as a matrix form when 
restricted to the degree 1 part of $S$.) 
So no power of any $\phi_i$ is a reflection. However,
\[ \psi = \phi_1\phi_2\phi_3 = \diag(\xi^{-2},1,\xi^2,-1)\]
and so $\psi^3=\diag(1,1,1,-1)$. That is, $\psi^3$ is a 
reflection. By an easy computation, $\mff_i=2$ for
all $i=1,2,3,4$. As a consequence, 
\[x_1^2 x_2^2x_3^2x_4^2=(\oj_S)^2=(\oa_S)^2=
(\pg_S)^2=\od_S\]
up to nonzero scalars. Using the results stated in the 
introduction, we have
\begin{enumerate}
\item[(1)]
$(S,O)$ does not satisfy Auslander's Theorem as $\mff_1\neq 1$
[Theorem~\ref{zzthm0.4}(6)],
\item[(2)]
$S$ is not Calabi--Yau as $\pg_S\not\in Z$ 
[Theorem~\ref{zzthm0.10}(4)], and
\item[(3)]
$Z$ is not Gorenstein as $\oj_S \; \pg_S\not\in Z$
[Theorem~\ref{zzthm0.11}(4)].
\end{enumerate}
\end{example}

We will shortly give necessary and sufficient conditions on 
the entries $b_{ij}$ of $B$ for $O$ to be small when $n = 4$. 
The following result gives a necessary condition for any $n$, 
and is of independent interest.

\begin{proposition}\label{zzpro3.4}
Assume Hypothesis~\ref{zzhyp0.3}.
If $O$ is small, then $\ell \mid \pf(B)$.
\end{proposition}
\begin{proof}
When $n$ is odd, $\pf(B) = 0$, and so the statement is 
vacuously true. Now suppose $n$ is even. We will prove the 
contrapositive: if $\pf(B)\ne 0$ in $\overline{\ZZ}$, then 
$O$ contains a reflection. 

Recall the adjugate matrix $\adj(B)$ is defined by the 
property
\[ \adj(B)B=B\adj(B)=\det(B)I.\]
There is a Pfaffian version of this for skew symmetric $B$ 
(with $n$ even), defined by
\[ \Padj(B)B=B\Padj(B)=\pf(B)I\]
\cite[Corollary 1, p.46]{Ba}. If $\pf(B)\not\equiv 0 \pmod{\ell}$, 
we can use the above to solve the equation
$B \by=\pf(B)\be_i$. For any $i$, 
\[ \by =\Padj(B)\be_i.\]
By the discussion preceding equation \eqref{E3.0.1}, this 
produces a reflection in $O$.
\end{proof}

\begin{proposition}\label{zzpro3.5}
Assume Hypothesis~\ref{zzhyp0.3} and let $n = 4$. Then $O$ 
is small if and only if the following two conditions hold:
\begin{enumerate}
\item[(1)]
$\pf(B) \equiv 0 \pmod{\ell}$ and
\item[(2)] 
there does not exist an index $j$ and an integer $k$ such 
that $k b_{ij} \equiv 0 \pmod{\ell}$ for all but one $i$.
\end{enumerate}
\end{proposition}
\begin{proof}
As discussed at the beginning of this section, the smallness 
of $O$ is equivalent to equation~\eqref{E3.0.1} having no 
solutions.

The previous proposition shows that if $O$ is small, then 
$\pf(B) \equiv 0 \pmod{\ell}$, so condition (1) holds. 
Further, if there did exist an index $i$ and an integer $k$ 
such that $k b_{ij} \equiv 0 \pmod{\ell}$ for all but one 
$j$, then $\overline{B} (k \be_j) = kb_{ij} \be_i$, 
and so we have a solution to \eqref{E3.0.1}. Hence, if $O$ is 
small then conditions (1) and (2) hold.

Conversely, suppose that conditions (1) and (2) hold. We 
wish to show that \eqref{E3.0.1} does not have any solutions.
Without loss of generality, suppose to the contrary that 
there is a solution, say $\bu$, when $i = 4$, so that $\lambda\be_4=\overline{B\bu}$. Multiplying both sides by $\Padj(B)$ gives
\[ \lambda\overline{\Padj(B)}\be_4=\overline{\Padj(B)Bu}=\overline{\pf(B)I\bu} = 0.\]
Using \cite[Definition 1, p.46]{Ba}, one easily computes that \[ \lambda\overline{\Padj(B)}\be_4=\lambda\overline{(-b_{23},b_{13},-b_{12},0)^t}.\]

Now since $b_{12}\lambda \equiv b_{13}\lambda \equiv 0 
\pmod{\ell}$, condition (2) implies that 
$b_{14} \lambda \equiv 0 \pmod{\ell}$. Similarly, since 
$b_{12} \lambda \equiv b_{23} \lambda \equiv 0 \pmod{\ell}$, 
then $b_{24} \lambda \equiv 0 \pmod {\ell}$. And since 
$b_{14}\lambda \equiv b_{24}\lambda \equiv 0 \pmod {\ell}$, 
then $b_{34} \lambda \equiv 0 \pmod{\ell}$. But since 
$\lambda \neq 0$, this implies that $\gcd(b_{ij},\ell) \neq 1$, 
which is a contradiction.
\end{proof}

Now we are ready to prove Theorem~\ref{zzthm0.5}.

\begin{proof}[Proof of Theorem~\ref{zzthm0.5}]
Note that the case $n=2$ is trivial. The case $n=3$ follows 
from Proposition~\ref{zzpro3.2} while the $n=4$ case follows 
from Proposition~\ref{zzpro3.5}.
\end{proof}

\begin{remark}\label{zzrem3.6}
We remark that condition (2) in Proposition~\ref{zzpro3.5} 
is equivalent to the condition that for all $1 \leq i \leq 4$, 
no  $\phi_i^k$ is a reflection. Hence, in both the cases 
$n = 3$ and $n = 4$, the smallness of $O$ is equivalent to 
the Pfaffian being $0$ (which is automatic when $n = 3$) and 
no power of a generator $\phi_i$ being a reflection. In 
Example~\ref{zzex3.3}, the Pfaffian of $B$ was nonzero.
\end{remark}

\begin{remark}\label{zzrem3.7}
Let $R$ be an AS regular algebra satisfying $\gldim(R)\geq 2$ 
and $G$ a finite subgroup of $\Autgr(R)$. We say $R^G$ has 
\emph{graded isolated singularities} if $\p(R,G) = \GKdim(R)$. 

Consider the case that $R=\SP$ on $n$ variables and $G=O$ as above. 
By \cite[Lemma 5.4]{BHZ2}, it suffices to compute $\p(A,G)$ 
where $A=\kk[x_1,\hdots,x_n]$. But in this setting, $\p(A,G)=n$ 
is equivalent to $G$ acting freely on $A_1\backslash\{0\}$
(see the references given in \cite[p.4320]{CYZ}). It is clear 
that this fails since, for example, $\phi_i(x_i)=x_i$, Hence, 
choosing any nontrivial $\phi_i$ (one of which must exist if 
$R$ is noncommutative), shows that $G$ \emph{does not} act 
freely on $A_1$. Thus $\ZSP$ does not have isolated singularities. 

On the other hand, there are other AS regular algebras $A$ such 
that $Z(A)$ has an isolated singularity \cite{CGWZ2}.
\end{remark}

\section{Regular center}\label{zzsec4}

Throughout this section we assume Hypothesis~\ref{zzhyp0.3}. 
We consider the question of determining when the center of 
$S$ is regular (equivalently, $Z$ is a polynomial ring). 

Recall that $\overline{B}$ is the matrix obtained from 
$B$ by reduction mod $\ell$. Let $\overline{K}$ denote 
the kernel of $\overline{B}$ and $K$ be its inverse image 
in $\ZZ^n$. For $i=1,\ldots,n$, denote by $K_i \subseteq \ZZ$ 
the projection of $K$ onto its $i$th component. If $p$ is 
a prime number, then let $\ZZ_{(p)}$ denote the localization 
of $\ZZ$ at the prime ideal $(p)$. If $M$ is a $\ZZ$-module 
and $m \in M$, we use the notation $m \otimes 1$ to denote 
the image of $m$ in $M \otimes \ZZ_{(p)}$.

\begin{lemma}\label{zzlem4.1}
Assume Hypothesis~\ref{zzhyp0.3}.
\begin{enumerate}
\item[(1)] 
Let $x_1^{u_1}\cdots x_n^{u_n}$ be a monomial in $S$. Write 
$\bu = (u_1, \dots, u_n)^T$ and $x^{\bu} := x_1^{u_1}\cdots 
x_n^{u_n}$. Then $x^\bu$ is central if and only if $\bu \in K$.
\item[(2)] 
We have $Z=\kk[x_1^{\mff_1},\ldots,x_n^{\mff_n}]$ if 
and only if $\mff_i\be_i\in K$ for each $i=1,\ldots,n$. 
Equivalently, $\mff_i\be_i \otimes 1 \in K\otimes 
\ZZ_{(p)}$ for every prime $p\mid \ell$ and $i=1,\ldots,n$. 
\end{enumerate}
\end{lemma}
\begin{proof}
(1) For any $1 \leq i \leq n$ we have
\begin{align*}
x_i\left(x_1^{u_1}\cdots x_n^{u_n}\right)
&= \xi^{b_{i1}u_1+\cdots+b_{in}u_n}
   \left(x_1^{u_1}\cdots x_n^{u_n}\right) x_i\\
&= \xi^{(B \bu)_i} \left(x_1^{u_1}\cdots x_n^{u_n}\right) x_i.
\end{align*}
Hence, $x^\bu$ is central if and only if $\xi^{(B\bu)_i}=1$ 
for all $i$ if and only if $\bu \in K$. 

(2) By part (1) and definition, $\mff_i=\gcd(K_i)$. If 
$\mff_i\be_i\in K$, then by part (1), we have 
$x_i^{\mff_i}\in Z$. Suppose that $x^\bu \in Z$, or 
equivalently, $\bu\in K$. Then by the definition of 
$\mff_i$ (see \eqref{E0.3.1}), there exist integers $a_i$ 
such that $\mff_ia_i=u_i$, so that 
$(x_1^{\mff_1})^{a_1}\cdots (x_n^{\mff_n})^{a_n}=x^\bu$. 
Hence $Z\subseteq \kk[x_1^{\mff_1},\ldots,x_n^{\mff_n}]$. 
The converse is clear.
\end{proof}

\begin{proof}[Proof of Theorem~\ref{zzthm0.6}]
(1) $\Leftrightarrow$ (5): 
This is \cite[Theorem 5.5]{KKZ3}. 

(1) $\Leftrightarrow$ (4): 
This follows from \cite[Lemma 1.10]{KKZ5}.

(5) $\Rightarrow$ (3): 
By the proof of Proposition~\ref{zzpro2.6}, when $O=H$,
$Z=S^O=S^H$ is the subring of the form 
$\kk[x_1^{\mff_1},\cdots,x_n^{\mff_n}]$.

(3) $\Rightarrow$ (2) $\Rightarrow$ (1): These implications are clear.

(3) $\Leftrightarrow$ (7): Since $O$ is generated by 
$\{\phi_i\}_{i=1}^n$, it is preserved under base field extension.
Further both assertions in (7) and (3) are preserved 
under base field extension. So we can assume that $\kk$ is 
algebraically closed, whence we can use Lemma~\ref{zzlem0.2}.
Given (3), it is easy to see that $\rk(S_Z)=\prod_{i=1}^n \mff_i$, and this is the order of $O$ by Lemma~\ref{zzlem0.2}. Conversely, 
$Z\subseteq \kk[x_1^{\mff_1},\cdots,x_n^{\mff_n}]:=B
\subseteq S$ by definition.
Then $\rk(S_B)\rk(B_Z)=\rk(S_Z)=|O|$
where the last equation is Lemma~\ref{zzlem0.2}.
Further, $O$ acts on $B$ with $Z=B^{O}$. Hence 
$Z$ is a direct summand of $B$ (see \cite[Lemma 1.11]{KKZ5} and \cite[Corollary 1.12]{Mo}). 
Since $\rk(S_B)=\prod_{i=1}^n \mff_i$, then 
$|O|=\prod_{i=1}^n \mff_i$ implies that $\rk(B_Z)=1$
or equivalently, 
$Z=B=\kk[x_1^{\mff_1},\cdots,x_n^{\mff_n}]$.

(3) $\Rightarrow$ (6): Since $O$ is generated by 
$\{\phi_1,\cdots,\phi_n\}$, there is a surjective map 
$\prod_{s=1}^{n} \langle \phi_s\rangle \to O$ and 
consequently,
\[ |O|\leq  \frac{\ell^n}{\prod_{s=1}^{n} 
\gcd\{b_{1s}, b_{2s},\cdots,b_{ns},\ell\}}.\]

Similar to the argument in
(3) $\Leftrightarrow$ (7), we can assume that $\kk$ is 
algebraically closed. Again by Lemma~\ref{zzlem0.2},
$|O|=\rk(S_Z)=\prod_{i=1}^{n} \mff_i$.
Since $x_i^{\mff_i}\in Z$, $p_{is}^{\mff_i}=1$
after applying $\phi_s$ to $x_i^{\mff_i}$. Equivalently,
$b_{is} \mff_i$ is divisible by $\ell$. Then
\[\gcd\{b_{i1},\cdots,b_{in},\ell\} \mff_i=
\gcd\{b_{i1}\mff_i,\cdots,b_{in}\mff_i,\ell\mff_i\} 
=\ell a\]
for some integer $a$. Therefore
$\mff_i\geq\frac{\ell}{\gcd\{b_{i1},\cdots,b_{in},\ell\}}$.
Thus 
\[ |O|\geq \frac{\ell^n} {\prod_{i=1}^n 
\gcd\{b_{i1},\cdots,b_{in},\ell\}}.\]
This proves equality. In fact we also obtain that
\begin{align}\label{E4.1.1}
\mff_i=\frac{\ell}{\gcd\{b_{i1},\cdots,b_{in},\ell\}}.
\end{align}

(6) $\Leftrightarrow$ (8): Since the order of $\phi_i$ is 
$\frac{\ell}{\gcd\{b_{i1},\cdots,b_{in},\ell\}}$, (6) 
implies that the surjective map $O\to \prod_{i=1}^n \langle 
\phi_i\rangle$ is bijective. The converse is similar.

(6) $\Rightarrow$ (2): 
Let $w_i=\frac{\ell}{\gcd\{b_{i1},\cdots,b_{in},\ell\}}$.
Then $w_i b_{is}$ is divisible by $\ell$. This implies that
$x_i^{w_i}\in Z$ for all $i$. Let $C=\kk[x_1^{w_1},
\cdots, x_n^{w_n}]$. So $\rk(S_C)=\frac{\ell^n}{
\prod_{i=1}^n \gcd\{b_{i1},\cdots,b_{in},\ell\}}$,
which is equal to $|O|$. Since $\rk(S_Z)=|O|$, 
we obtain that $\rk(Z_C)=1$.
Since $Z$ is Cohen--Macaulay \cite[Lemma 3.1]{JoZ}, then $C=Z$.

(9) $\Leftrightarrow$ (3): This is Lemma~\ref{zzlem4.1}(2).

When $Z$ is regular, it is Gorenstein. By Theorem~\ref{zzthm2.7},
$\oj_S$ is equal to $\fj_{S,O}$. By a decomposition like
\eqref{E2.7.1}, one sees that $\oa_S$ is equal to 
$\fa_{S,O}$. 
\end{proof}

\begin{remark}\label{xxrem4.2}
We now describe an idea (or rather an algorithm) for testing 
regularity and Gorensteinness of $Z$ that works for any 
$n$. We divide it into several steps.
\begin{enumerate}
\item[(1)]
Recall that if $M$ is any matrix over a PID, 
then there exists a diagonal matrix $D$ and invertible 
matrices $L,R$ such that $D=LMR$. The matrix $D$ is the 
\emph{Smith normal form} of $M$. We will apply this to
the matrix $B$.
\item[(2)]
By definition, $K$ is the preimage of $\overline{K}$ 
where $\overline{K}$ is the kernel of the map 
$L_{\overline{B}}: \overline{\ZZ}^n\to \overline{\ZZ}^n$
(as a left multiplication by the matrix $\overline{B}$). 
We also use $\overline{B}$ to denote this (right) 
$\ZZ$-module endomorphism $L_{\overline{B}}$ if no confusion 
occurs. By definition there is a short exact sequence
$$0\to \ell (\ZZ^n)\to K\to \overline{K}\to 0.$$
As a consequence, $K$ contains $\ell \be_i$ for each
$i$. We may consider $\{\ell\be_i\}_{i=1}^n$ as a subset of 
a generating set of $K$. If necessary we also view 
$\overline{B}$ as the composition map
$\ZZ^n\to \overline{\ZZ}^n\xrightarrow{L_{\overline{B}}} 
\overline{\ZZ}^n.$ In this setting, $K$ is the 
kernel of $\overline{B}$. 
\item[(3)]
Since $\ZZ$ is integrally closed and $K$ is finitely generated, 
the conditions in Lemma~\ref{zzlem4.1}(2)
can be checked locally (at prime 
$p$ for all $p\mid \ell$), and to do this it is convenient to have a 
generating set for $K_{(p)}:=K\otimes \ZZ_{(p)}$. To this end, we 
compute the kernel of the map 
$\overline{B}_{(p)}:=\overline{B}\otimes \ZZ_{(p)}$ for 
each prime $p \mid \ell$, in other words, we produce a generating 
set for $K_{(p)}$, which can be glued together to get a 
generating set for $K$. 
\item[(4)]
Recall that for an integer $m$ and a prime $p$, $\nu_p(m)$ 
denotes the maximal integer $a$ such that $p^a \mid m$. 
Recall from Hypothesis~\ref{zzhyp0.3} that 
$B=(b_{ij})_{i,j}$ is an $n \times n$ skew symmetric 
matrix over $\ZZ$. We further introduce some notations: 
fix a prime $p$ dividing $\ell$ and let 
\begin{equation}
\label{E4.2.1}
N=\nu_p(\ell), \quad
\alpha_{ij}=\min\{N,\nu_p(b_{ij})\},
\quad  {\text{and}} \quad 
\alpha = \min\{N,\nu_p(\pf(B))\}.
\end{equation}
\item[(5)]
For each $p\mid \ell$, $\overline{\ZZ}\otimes \ZZ_{(p)}\cong 
\ZZ/(p^{N})$ and $K_{(p)}$ is the kernel of the left multiplication 
map
$$\overline{B}_{(p)}: \ZZ_{(p)}^n \to (\ZZ/(p^{N}))^n.$$
\item[(6)]
By part (2), for each $p\mid \ell$, $p^N \be_i\in K_{(p)}$.
So we are interested in other generators in $K_{(p)}$. 
In other words, we are interested in generators of 
$\overline{K}_{(p)}$. 
\item[(7)]
Given the Smith normal form $D = LBR$ for $B$, the 
equation $\overline{B}_{(p)} \bu=0$ is equivalent to
$B \bu\equiv 0 \pmod{p^N}$ and is equivalent to 
$D\bv \equiv 0 \pmod{p^N}$ where $\bv = R^{-1}\bu$. Hence, 
to compute the kernel of $\overline{B}_{(p)}$, we compute 
the kernel of $D$ mod $p^N$ (or equivalently, the kernel of 
the map $\overline{D}_{(p)}$) and apply $R$ to the generators 
to obtain a set of generators $\{\bu_1, \dots, \bu_r\}$ 
for the kernel of $\overline{B}_{(p)}$. If we let 
$\mff_{p,j} = \gcd_{1 \leq i \leq r}((\bu_i)_j)$ (computed 
in $\ZZ_{(p)}$),
then $\mff_{p,j}$ is of the form 
$p^a u$ where $a$ is a nonnegative integer and $u$ is a unit
in $\ZZ_{(p)}$. So we can write $\mff_{p,j}=p^a$, which is
called the {\it standard form} of $\mff_{p,j}$. Using the standard forms,
one sees that $\mff_{j}$ is the lcm of the $\mff_{p,j}$ as $p$ 
runs over the prime divisors of $\ell$.
Since $\mff_{p,j}=\mff_{j}$
in $\ZZ_{(p)}$, we will also use $\mff_{j}$ for $\mff_{p,j}$ 
in the middle of the proofs.
\item[(8)]
By Lemma~\ref{zzlem4.1} (resp. Lemma~\ref{zzlem5.1} in the next section), 
to determine if the center $Z$ is regular (resp. Gorenstein), 
it is enough check whether $\mff_i \be_i \in K$ (resp. 
$(\mff_1, \dots, \mff_n)^T \in K$). 
\item[(9)]
For the argument below, we will fix a prime divisor $p$ 
of $\ell$. Throughout the rest of this section 
we will use the convention introduced in this remark.
\end{enumerate}
\end{remark}

Using the ``algorithm'' discussed above, we give explicit 
conditions equivalent to the regularity of $Z$ in terms of 
the parameters $b_{ij}$ in the cases $n=3$ and $n=4$.

\begin{proposition}\label{zzpro4.3}
Assume Hypothesis~\ref{zzhyp0.3} and retain the convention
introduced in \eqref{E4.2.1}. Assume, without loss of 
generality, that $\alpha_{12} = \min\{\alpha_{ij}\}=0$.
\begin{enumerate}
\item[(1)] 
If $n=3$, then $K_{(p)}$ is generated as a 
$\ZZ_{(p)}$-module by $p^N\be_i$ for $i=1,2$ and 
\[ \frac{1}{b_{12}}
\begin{bmatrix}b_{23}\\-b_{13} \\ b_{12}\end{bmatrix}.\]
\item[(2)] 
If $n=4$, then $K_{(p)}$ is generated as a 
$\ZZ_{(p)}$-module by $p^N\be_i$ for $i=1,2$ and 
\[
p^{N-\alpha}
\begin{bmatrix}
b_{23}/b_{12}\\-b_{13}/b_{12} \\ 1 \\ 0 
\end{bmatrix}, \quad
p^{N-\alpha}
\begin{bmatrix}
b_{24}/b_{12}\\-b_{14}/b_{12} \\ 0 \\ 1
\end{bmatrix}.\]
\end{enumerate}
\end{proposition}

\begin{proof}
(1) Let $n=3$. The Smith normal form $D=L B R$ of $B$ over 
the ring $\ZZ_{(p)}$ is
\[
D = \begin{bmatrix}
    b_{12}& 0 & 0 \\
    0 & -b_{12} & 0\\
    0 & 0 & 0
\end{bmatrix}, \quad
L = \begin{bmatrix}
    1 & 0 & 0 \\
    0 & 1 & 0 \\
    b_{23}/b_{12} & -b_{13}/b_{12} & 1
\end{bmatrix}, \quad
R = \begin{bmatrix}
    0 & 1 & b_{23}/b_{12} \\
    1 & 0 & -b_{13}/b_{12} \\
    0 & 0 & 1
\end{bmatrix}.
\]
Applying the argument in Remark \ref{xxrem4.2}(7),
the kernel of $D_{(p)}$ is generated, as a $\ZZ_{(p)}$-module, 
by $p^{N}\be_1$, $p^{N}\be_2$, $\be_3$. Applying $R$ to 
these gives the stated generators.

(2) Let $n=4$. The Smith normal form of $B$ is given by
\begin{align*}
    D &= \begin{bmatrix}
         b_{12}& 0 & 0 & 0 \\
         0 & -b_{12} & 0 & 0\\
         0 & 0 & \frac{1}{b_{12}}\pf(B) & 0 \\
         0 & 0 & 0 & -\frac{1}{b_{12}}\pf(B) 
    \end{bmatrix},\\
    L &= \frac{1}{b_{12}}\begin{bmatrix}
         b_{12} & 0 & 0 & 0 \\
         0 & b_{12} & 0 & 0 \\
         b_{23} & -b_{13} & b_{12} & 0 \\
         b_{24} & -b_{14} & 0 & b_{12} 
    \end{bmatrix},\\
    R &= \frac{1}{b_{12}}\begin{bmatrix}
         0 & b_{12} & b_{24} & b_{23} \\
         b_{12} & 0 & -b_{14} & -b_{13}  \\
         0 & 0 & 0 &  b_{12} \\
         0 & 0 & b_{12} & 0  
    \end{bmatrix}.
\end{align*}
Using similar reasoning as above, the kernel of $D_{(p)}$ 
is generated by $p^{N}\be_1$, $p^{N}\be_2$, 
$p^{N-\alpha}\be_3$, $p^{N-\alpha}\be_4$. Again, 
applying $R$ to these gives the stated generators. 
\end{proof}

We will write $\mff_i \be_i$ for $\mff_i\be_i 
\otimes 1 \in K_{(p)}:=K\otimes \ZZ_{(p)}$ if no 
confusion occurs.

\begin{proposition}
\label{zzpro4.4}
Keep the assumptions of Proposition~\ref{zzpro4.3}. 
If $n=3$, then $\mff_i\be_i\in K_{(p)}$ for all 
$i=1,2,3$ if and only if $\alpha_{23}=\alpha_{13}=N$.
\end{proposition}

\begin{proof}
Using the generators of $K_{(p)}$ from 
Proposition~\ref{zzpro4.3} and the definition of 
$\mff_i$ (c.f. \eqref{E0.3.1}), we have, up to units
in $\ZZ_{(p)}$, 
\begin{align}\label{E4.4.1}
    \mff_1 &= b_{23}/b_{12}, \qquad
    \mff_2 = b_{13}/b_{12}, \qquad
    \mff_3 = 1.
\end{align}
(Strictly speaking these should be $\mff_{p,1}, \mff_{p,2}$
and $\mff_{p,3}$ respectively.)
Suppose $\mff_i\be_i \in K_{(p)}$ 
for all $i$. In particular, for $i=3$, there exists 
$\lambda_1, \lambda_2, \lambda_3 \in \ZZ_{(p)}$ such that 
\begin{align*}
    \begin{bmatrix}
        \lambda_1 p^{N} \\ 0 \\ 0
    \end{bmatrix} + 
    \begin{bmatrix}
        0 \\ \lambda_2 p^{N} \\ 0
    \end{bmatrix} + 
    \begin{bmatrix}
        \lambda_3 b_{23}/b_{12} \\ 
				-\lambda_3 b_{13}/b_{12} \\ 
				\lambda_3
    \end{bmatrix} &= 
    \begin{bmatrix}
        0 \\ 0 \\ 1
    \end{bmatrix}.
\end{align*}
Hence $\lambda_3=1$. The first component gives
$\lambda_1 = -b_{23}p^{-N}/b_{12}$, so 
$\nu_p(\lambda_1) = \alpha_{23}-N$ and a necessary and sufficient 
condition for $\lambda_1\in\ZZ_{(p)}$ is $\alpha_{23}=N$. 
Similarly we get $\alpha_{13}=N$, and this proves the forward 
implication. The converse is clear.
\end{proof}

\begin{proposition}\label{zzpro4.5}
Keep the assumptions of Proposition~\ref{zzpro4.3}. 
If $n=4$, then $\mff_i\be_i\in K_{(p)}$ 
for all $i=1,\ldots,4$ if and only if 
\begin{align}
\label{E4.5.1}
\alpha_{ij} &\ge \alpha
\end{align}
for $(i,j) = (1,3), (1,4), (2,3), (2,4)$. 
\end{proposition}
\begin{proof}
Using the generators of $K \otimes \ZZ_{(p)}$ from 
Proposition~\ref{zzpro4.3} and the definition of $\mff_i$, 
we have
\begin{align}
\label{E4.5.2}
\mff_1 &=\gcd\left( p^{N},p^{N-\alpha} \frac{b_{23}}{b_{12}}, 
        p^{N-\alpha}\frac{b_{24}}{b_{12}} \right) 
				=\gcd\left(p^{N-\alpha} \frac{b_{23}}{b_{12}}, 
				p^{N-\alpha}\frac{b_{24}}{b_{12}} \right) \\
\mff_2 &=\gcd\left( p^{N},p^{N-\alpha} \frac{b_{13}}{b_{12}}, 
        p^{N-\alpha}\frac{b_{14}}{b_{12}} \right) 
				=\gcd\left(p^{N-\alpha} \frac{b_{13}}{b_{12}}, 
				p^{N-\alpha}\frac{b_{14}}{b_{12}} \right) 
				\nonumber\\
\mff_3 &= \mff_4 = p^{N-\alpha}.  \nonumber
\end{align}
Suppose $\mff_i\be_i \in K_{(p)}$ for 
all $i$. In particular, for $i=3$, there exists 
$\lambda_1,\ldots,\lambda_4 \in \ZZ_{(p)}$ such that 
\begin{align*}
\begin{bmatrix}
        \lambda_1 p^{N} \\ \lambda_2 p^{N} \\ 0 \\ 0
\end{bmatrix} + 
    p^{N-\alpha}
\begin{bmatrix}
        \lambda_3 b_{23}/b_{12} \\ -\lambda_3 b_{13}/b_{12} 
				\\ \lambda_3 \\ 0
\end{bmatrix} + 
    p^{N-\alpha}
\begin{bmatrix}
        \lambda_4 b_{24}/b_{12} \\ -\lambda_4 b_{14}/b_{12} 
				\\ 0 \\ \lambda_4
\end{bmatrix}
    &= 
\begin{bmatrix}
        0 \\ 0 \\ p^{N-\alpha} \\ 0
\end{bmatrix}.
\end{align*}
Hence $\lambda_4=0$ and $\lambda_3=1$. The first component gives 
$\lambda_1  =-p^{-\alpha}b_{23}/b_{12}$, so 
$\nu_p(\lambda_1)=-\alpha+\alpha_{23}$. So 
$\lambda_1\in\ZZ_{(p)}$ if and only if $\alpha_{23} \ge \alpha$. 
The same argument on the second component gives 
$\alpha_{13}\ge \alpha $. We can perform the same calculations 
for $i=4$ to obtain $\alpha_{14}\ge \alpha$ and 
$\alpha_{24}\ge \alpha$. 

Conversely, assume the inequalities \eqref{E4.5.1} hold. Then 
$N\le N-\alpha + \alpha_{ij}$ so $\mff_1 = \mff_2=p^{N}$. Hence 
$\mff_i\be_i\in K_{(p)}$ for $i=1,2$. 
Next we take the third generator of $K_{(p)}$ 
(c.f. Proposition~\ref{zzpro4.3}) and subtract from it multiples 
of the first two
\[
     \mff_3\be_3 = p^{N-\alpha}\begin{bmatrix}
     b_{23}/b_{12}\\-b_{13}/b_{12} \\ 1 \\ 0
\end{bmatrix}-\beta_1 p^{N}\be_1+\beta_2 p^{N}\be_2
\]
where $\beta_1=p^{-\alpha}b_{23}/b_{12}$ and 
$\beta_2= p^{-\alpha}b_{13}/b_{12}$. The inequalities 
\eqref{E4.5.1} can be rearranged so that 
$\alpha_{ij}- \alpha\ge 0$. The left hand side of this 
inequality is $\nu_p(\beta_k)$ for appropriate $i,j,k$. 
Hence $\beta_1,\beta_2\in\ZZ_{(p)}$, so 
$\mff_3\be_3\in K_{(p)}$. Similar 
computations with the fourth generator of $K$ yield 
$\mff_4\be_4\in K_{(p)}$. This 
completes the proof.  
\end{proof}

Next we globalize the above results which are local 
at $p$. First we need a technical lemma.

\begin{lemma}\label{zzlem4.6}
Let $r_1,\ldots, r_k$ be roots of unity with orders $o_1,\ldots, o_k$. Let $s=o_1\cdots o_k$ and $\zeta$ be a primitive $s$-th root of unity. Then the orders are pairwise coprime if and only if there exist integers $n_1,\ldots, n_k$ such that $\gcd(n_i, s)=1$ and 
\begin{align*}
    r_i &= \zeta^{n_i o_1\cdots \hat{o}_i \cdots o_k}
\end{align*}
\end{lemma}
\begin{proof}
We prove the forward direction only, since the reverse implication is immediate. Since $r_i$ has order $o_i$, it is a primitive $o_i$-th root of unity. Now $\zeta^{o_1\cdots \hat{o}_i \cdots o_k}$ is also a primitive $o_i$-th root of unity, so there exists an integer $m_i$, such that $m_i$ and $o_i$ are coprime, and
\begin{align*}
    r_i &= \zeta^{m_i s_i}
\end{align*}
where $s_i=o_1\cdots \hat{o}_i \cdots o_k$. The integers $m_i$ and $s$ may not be coprime, but we can choose an integer $j$ such that $n_i:=m_i+jo_i$ and $s$ are coprime. Firstly, since $o_i$ and $s_i$ are coprime, we have $\left\{m_i+jo_i\right\}_{j=0}^{s_i}=\ZZ/s_i\ZZ$ as sets. Hence there exists an integer $j$ such that $n_i$ is coprime to $s_i$. But $m_i$ and $o_i$ are assumed to be coprime, so $n_i$ and $o_i$ are coprime, which means $n_i$ and $s$ are coprime. Finally, 
\[
    \zeta^{n_i s_i} = \zeta^{(m_i + jo_i) s_i} = \zeta^{m_i s_i + js } = \zeta^{m_i s_i} = r_i
\]
and we are done.
\end{proof}

\begin{corollary}\label{zzcor4.7}
Assume Hypothesis~\ref{zzhyp0.3} and let $n=3$. 
Then $Z$ is regular if and only if the orders 
of $p_{12}$, $p_{13}$ and $p_{23}$ are pairwise 
coprime. Equivalently, there exist pairwise 
coprime integers $a,b,c\ge 1$ such that 
$p_{12}=\xi^{ab}$, $p_{13}=\xi^{ac n_0}$ and 
$p_{23}=\xi^{bcm_0}$ where $\xi$ is a primitive 
$\ell$th root of unity with $\ell=abc$ and 
$n_0,m_0$ are coprime to $\ell$. 
\end{corollary}
\begin{proof}
For each prime $p$ dividing $\ell$, we have 
$\nu_p(o(p_{ij}))=N-\alpha_{ij}$. The orders of 
$\{p_{ij}\}_{i<j}$ being pairwise coprime is equivalent to 
at least two of $\{\alpha_{ij}\}_{i<j}$ being equal to $N$. 
By Proposition~\ref{zzpro4.4} and Lemma~\ref{zzlem4.1}, 
this occurs if and only if $Z$ regular. The last 
statement follows from Lemma~\ref{zzlem4.6}.
\end{proof}

\begin{corollary}\label{zzcor4.8}
Assume Hypothesis~\ref{zzhyp0.3} and let $n=4$. 
Let $\rho=\gcd(\ell, \pf(B))$, $c_{ij} = \gcd(b_{ij}, \rho)$, 
$\omega$ be a primitive $\rho$th root of unity, and 
$q_{ij}=\omega^{c_{ij}}$. Then $Z$ is regular if and only 
if the orders of $\{q_{ij}\}_{i<j}$ are pairwise coprime. 
Equivalently, there exist pairwise coprime integers 
$o_{ij}\ge 1$ such that
\[ q_{ij} = \omega^{n_{ij}\hat{o}_{ij}}\]
where $(n_{ij}, \rho)=1$ and $\prod_{i<j} o_{ij}
=\rho$ and $\hat{o}_{ij} = \rho / o_{ij}$. 
\end{corollary}

\begin{proof}
We first show that the conditions on $\alpha_{ij}$ in 
Proposition~\ref{zzpro4.5} imply that $\alpha_{34}\ge 
\alpha$ as well. Rearranging the equation for the Pfaffian 
gives 
\[ b_{12}b_{34} = \pf(B)+b_{13}b_{24}-b_{14}b_{23}.\]
Taking $p$-valuations and noting that $\alpha_{12}=0$ gives 
the inequality.

Let $p$ be a prime factor of $\rho$. Then
$\nu_p(o(q_{ij})) = \nu_p(\rho) - \nu_p(c_{ij})$. We also have 
$\nu_p(\rho) =\alpha$ and
$\nu_p(c_{ij}) = \min\{\alpha_{ij}, \alpha\}$. 
Proposition~\ref{zzpro4.5} and Lemma~\ref{zzlem4.1} say 
that $Z$ is regular if and only if $\alpha_{ij}\ge\alpha$ 
for $(i,j)\ne (1,2)$. These conditions imply  
$\nu_p(o(q_{ij}))=0$ except $(i,j)=(1,2)$. Hence $o(q_{ij})$ 
are pairwise coprime.

Conversely, we get that $\nu_p(o(q_{ij}))=0$ for all 
$(i,j)\ne (1,2)$. This means $\min\{\alpha_{ij}, \alpha\} 
= \alpha$ implying $\alpha_{ij}\ge \alpha$ 
for such $(i,j)$. 

The last statement follows from Lemma~\ref{zzlem4.6}.
\end{proof}

We finish this section with the proofs of Theorems~\ref{zzthm0.7} 
and \ref{zzthm0.9}.

\begin{proof}[Proof of Theorem~\ref{zzthm0.7}]
In the $n=2$ case, $Z=\kk[x_1^\ell,x_2^\ell]$ so the 
result is clear. The result in the $n=3$ case follows from 
Corollary~\ref{zzcor4.7}, while the result in the $n=4$ case is 
due to Corollary~\ref{zzcor4.8}.
\end{proof}

\begin{proof}[Proof of Theorem~\ref{zzthm0.9}] 
If the orders of $\{p_{ij}\}_{i<j}$ are pairwise coprime, 
then for any prime $p\mid \ell$, the matrix $\overline{B}_{(p)}$ 
has rank $2$. In particular, $\overline{B}_{(p)}$ is similar to the 
block diagonal matrix formed from 
\begin{align*}
    \begin{bmatrix}
         0 & b \\
         -b& 0
    \end{bmatrix}
\end{align*}
and a zero matrix of the appropriate dimension. Hence 
$K_{(p)}$ is generated by 
$p^{N-\nu_p(b)}\be_1,p^{N-\nu_p(b)}\be_2$ and $\be_3,\ldots,\be_n$. 
This shows that $\nu_p(\mff_i)=N-\nu_p(b)$ for $i=1,2$ and 
$\nu_p(\mff_j)=0$ for $j=3,\ldots,n$. The result follows from 
Lemma~\ref{zzlem4.1}(2).

Part (1) of the consequence follows from \eqref{E4.1.1}.
By definition, $\ell=\prod_{i<j} o(p_{ij})$. Then 
part (2) of the consequence follows from part (1) and 
Theorem \ref{zzthm0.6}(7). 
\end{proof}

\section{Gorenstein center}\label{zzsec5}

Throughout this section we assume Hypothesis~\ref{zzhyp0.3}. 
Here we study the question of when the center of $S$ is 
Gorenstein. 

\begin{proof}[Proof of Theorem~\ref{zzthm0.11}]
(1) $\Leftrightarrow$ (2): This is \cite[Theorem 0.2]{KKZ4}.

(2) $\Leftrightarrow$ (3): 
Let $\overline{\phi_i}$ denote the image of $\phi_i$ in $O/H$. 
Since $S^H=\kk[x_1^{\mff_1}, \cdots, x_n^{\mff_n}]$ (see the 
proof of Proposition~\ref{zzpro2.6}) and $\overline{\phi_i} 
\in O/H$ acts on $S^H$ as a diagonal map for each $i$, 
\eqref{E1.4.2} implies that
\[
\hdet_{O/H}(\overline{\phi_i})=\overline{\phi_i}
(x_1^{\mff_1}\cdots x_n^{\mff_n}) 
(x_1^{\mff_1}\cdots x_n^{\mff_n})^{-1}
=\prod_{s=1}^{n} p_{is}^{\mff_s}.
\]
Since $\oj_s \pg_s = x_1^{\mff_1} \cdots x_n^{\mff_n}$, we 
see that (2) $\Leftrightarrow$ (3).

(3) $\Leftrightarrow$ (4): This is clear as an element 
$f \in Z$ if and only if $\phi_i(f) = f$ for all $i$.
\end{proof}

Next we prove Theorem~\ref{zzthm0.10}.

\begin{proof}[Proof of Theorem~\ref{zzthm0.10}]
(2) $\Leftrightarrow$ (3): This is Lemma~\ref{zzlem1.5}(2). 

(1) $\Leftrightarrow$ (3):
The Nakayama automorphism $\mu$ of $S$ 
follows from \cite[Proposition 4.1]{LWW}
and it is easy to show that (3) holds if and only if $\mu=\id$. Equivalently, 
$S$ is Calabi--Yau.

(3) $\Leftrightarrow$ (4): This follows by an easy computation.

Suppose that $O$-action has trivial homological determinant. 
By \cite[Theorem 1.21]{Zhu22}, Auslander's Theorem holds for $(S, O)$.
Hence, all statements in Theorem~\ref{zzthm0.4} hold.
Since Auslander's Theorem holds for $(S, O)$, $H$ is trivial.
Then Theorem \ref{zzthm0.11}(2) holds. Hence all statements in 
Theorem~\ref{zzthm0.11} hold.
\end{proof}

Theorem~\ref{zzthm0.11}(3) gives us necessary and sufficient 
conditions for $Z=S^O$ to be Gorenstein. It is more convenient 
for us to express the condition in terms of the matrix $B$.

\begin{lemma}[$\equiv$ Theorem~\ref{zzthm0.11}(3)]\label{zzlem5.1}
Assume Hypothesis~\ref{zzhyp0.3}.
The center $Z$ of $S$ is Gorenstein if and only if the 
following equation holds in $\overline{\ZZ}^n$:
\[ \overline{B}(\mff_1,\ldots,\mff_n)^T = 0.\]
\end{lemma}

We already computed the $\{\mff_i\}$ for $n=3,4$ in the 
previous section, so we are ready to prove Theorem~\ref{zzthm0.12}.

\begin{proof}[Proof of Theorem~\ref{zzthm0.12}]
(1) The result for $n=2$ is well-known. 

(2) The result for $n=3$ is immediate from Lemma~\ref{lem.fhyper}.

(3) Let $\bv=\ell/\gcd(\pf(B),\ell) (v_1,\ldots,v_4)^T$. 
The image of this vector in $(\ZZ_{(p)})^4$ is a unit multiple of 
\[
    p^{N-\alpha} \begin{bmatrix}     \gcd(b_{23},b_{24},b_{34})\\\gcd(b_{13},b_{14},b_{34}) 
				\\ \gcd(b_{12},b_{14},b_{24}) \\ \gcd(b_{12},b_{13},b_{23})
    \end{bmatrix}.
\]
After relabelling if necessary, we can assume $b_{12}$ is a unit. 
The equation  
\[ b_{12}b_{34} = \pf(B) + b_{13}b_{24} - b_{14}b_{23} \]
shows that $\alpha_{34}\ge \alpha$, which means we can drop $b_{34}$ 
from the arguments in the first two $\gcd$s above. Comparing with 
\eqref{E4.5.2} shows that the $i$th component of $\bv$ has the 
same $p$-valuation as $\mff_i$. This holds for all prime factors 
$p$ of $\ell$ so we are done.
\end{proof}

We end this section by discussing some examples and showcase the subtleties in the results above.
We remark that for a commutative ring, the following set of implications hold:
\[ \text{regular $\Rightarrow$ hypersurface ring $\Rightarrow$ complete intersection $\Rightarrow$ Gorenstein}.\]
In this paper, we have focused on the regular and Gorenstein properties of the center of $\SP$, but it would be interesting to determine conditions equivalent to the center of $\SP$ being a hypersurface ring or a complete intersection. Recall that $Z$ is a (commutative) hypersurface ring if $Z\cong\kk[x_1,\hdots,x_n]/(f)$ for some homogeneous polynomial $f$. In this case, the Hilbert series satisfies $h_Z(t)=p(t)/q(t)$ where $p(t)$ is a cyclotomic polynomial.

\begin{example}
Set $n=3$.

(1) By Theorem \ref{zzthm0.10}, if $S$ is Calabi-Yau, then
\[ Z = \kk[x_1^\ell, x_2^\ell, x_3^\ell, x_1x_2x_3]/(x_1^\ell x_2^\ell x_3^\ell - (x_1x_2x_3)^\ell)\]
is a hypersurface ring.

(2) 
Let $\ell>1$ and consider the following $B$ matrix:
\[
B = \begin{pmatrix}
0 & 0 & 1 \\ 
0 & 0 & -1 \\
-1 & 1 & 0
\end{pmatrix}.
\]
By Lemma \ref{lem.fhyper} and the subsequent discussion, $\mff = (1,1,\ell)$ and $g = (\ell,\ell,\ell)$.
It follows that $Z$ is generated by $\{x^\ell, y^\ell, z^\ell, xy\}$ and so
\[ Z \cong \kk[y_1,y_2,y_3,y_4]/(y_4^\ell - y_1y_2).\]
Thus, $Z$ is a hypersurface ring that is not Calabi--Yau.

(3)
Let $\ell=24$ and consider the following matrix:
\[ B_k = \begin{pmatrix}0 & 4 & 6 \\ -4 & 0 & k \\ -6 & -k & 0\end{pmatrix}\]
where $k=3$ or $9$.
In either case, Lemma \ref{lem.fhyper} gives $\mff=(3,6,4)$ and $g=(12,24,8)$.
By Lemma \ref{zzlem5.1}, $S$ is Gorenstein when $k=9$ and non-Gorenstein when $k=3$. Suppose $k=9$, then $Z$ is generated by $x_1^{12}$, $x_2^{24}$, $x_3^8$, and $x_1^3x_2^6x_3^4$. 
One checks that the Hilbert series of $Z$ is 
\[ h_Z(t) = \frac{p(t)}{q(t)}=\frac{1+t^{13}+t^{18}+t^{31}}{(1-t^{12})(1-t^{24})(1-t^8)}.\]
It is clear that $p(t)$ is not cyclotomic, and thus $Z$ is not a hypersurface ring.
\end{example}

\section{Questions and comments}\label{zzsec6}

In this section we list some questions and comments related 
to the theorems given in this paper. Much of this paper has 
been devoted to the study of the algebras $\SP$. In addition 
to classifying those $\SP$ such that $\ZSP$ is regular or 
Gorenstein for $n \geq 5$, one could consider the following 
problems.

\begin{question}\label{zzque6.1}
\begin{enumerate}
\item[(1)]
What conditions on the PI skew polynomial ring $\SP$ are equivalent to the center $\ZSP$ being a hypersurface ring or a complete intersection ring?
\item[(2)]
Let $\SP$ be the skew polynomial ring in $n$ variables. By 
Remark~\ref{zzrem3.7}, $0<\p(\SP,O)<n$. For each $0<i<n$, can 
we find some skew polynomial ring $S$ such that $\p(S, O)=i$?
\item[(3)] 
(Ken Goodearl) 
The matrix $B$ controls the PI degree of $S$ \cite{DCP}. Is there a 
direct computation of PI degree from the parameters $b_{ij}$? Is the PI degree related to the other invariants in this paper?
\item[(4)] 
(Colin Ingalls) 
In cases that the center $\ZSP$ is not Gorenstein, under what 
conditions on $\SP$ is it $\mathbb{Q}$-Gorenstein (see 
\cite{CCTIJKKLNOV})?
\item[(5)] The classification of skew polynomial rings up to isomorphism is known (see \cite{Ga}). Two skew polynomial rings are \emph{birationally equivalent} if their associated quotient division rings are isomorphic. What is the classification of skew polynomial rings up to birational equivalence?
\end{enumerate}
\end{question}

As we have mentioned, the algebras $\SP$ are viewed as a good testing ground for many 
problems in noncommutative invariant theory. It would 
be interesting to study the problems in this paper more 
generally.

\begin{question}\label{zzque6.3}
Let $A$ be a noetherian PI AS regular algebra.
\begin{enumerate}
\item[(1)]
Is there is a version of Corollary \ref{zzcor0.12}(1) or (2) for $A$?
\item[(2)]
Does (1) $\Leftrightarrow$ (3) in Theorem~\ref{zzthm0.10} hold for $A$?
Namely, is it true that the homological determinant of $\Oz(A)$ 
is trivial if and only if $A$ is Calabi--Yau?
\item[(3)] 
Can we define invariants $\oj_A$, $\oa_A$, and $\od_A$ such that 
they generalize $\oj_S$, $\oa_S$, and $\od_S$ and control 
properties of $A$ and its center?
\item[(4)] 
For the skew polynomial ring $\SP$, the ozone group $O$ acts 
on $\SP$ such that $\ZSP = \SP^O$. Is there a semisimple 
Hopf algebra $H$ acting on $A$ such that $Z(A)=A^H$?
\item[(5)] 
Suppose $A$ is generated in degree 1 and suppose that the center 
of $A$ is $\kk[c_1,\cdots, c_n]$ where $\deg c_i>1$ for every 
$i$. 
Does it hold that $\Aut(A)$ is affine?

Related to this is the notion of \emph{$\mathrm{LND}$-rigidity}. Let $\mathrm{LND}(A)$ denote the intersection of all kernels of locally nilpotent derivations of $A$.
Under the above hypotheses, does it hold that $\mathrm{LND}(A)=\{0\}$.
We note that both of these hold when $A$ is a skew polynomial ring.
\item[(6)]
Is there a unique mozone subring of $A$? That is, does 
$\Phi_Z(A)$ have a minimum element? 
\end{enumerate} 
\end{question}

Question~\ref{zzque6.3}(4) would be especially interesting 
in the case when $A$ is a (PI) Sklyanin algebra.

One source of interesting examples that may be useful for 
studying some of the above questions are graded noetherian 
down-up algebras.

Let $\alpha,\beta \in \kk$ with $\beta \neq 0$. A noetherian 
graded down-up algebra $A=A(\alpha,\beta)$ is generated as 
an algebra by $x$ and $y$ and subject to the relations
\[x^2 y -\alpha xyx -\beta yx^2=0=xy^2-\alpha yxy-\beta y^2x.\]
It is well-known that $A$ is AS regular, but is not isomorphic 
to any $\SP$. The group of graded algebra automorphisms of $A$ was 
computed by Kirkman and Kuzmanovich \cite[Proposition 1.1]{KK}. 
If $\beta\neq \pm 1$, then every graded algebra automorphism 
of $A$ is diagonal. By \cite[(1.5.6)]{LMZ}, the Nakayama 
automorphism of $A$ is determined by
\[\mu: x\mapsto -\beta x, \quad y\mapsto -\beta^{-1} y.\]
Let $\omega_1$ and $\omega_2$ be the roots of the 
characteristic equation
\[w^2-\alpha w-\beta=0\]
and let $\Omega_i=xy-\omega_i yx$ for $i=1,2$. It is easy to 
see that for $\{i,j\}=\{1,2\}$ we have
\[
x\Omega_i=\omega_j \Omega_i x, \quad
\omega_j y \Omega_i = \Omega_i y.
\]
Note that $A$ is PI if and only if both $\omega_1$ and $\omega_2$ 
are roots of unity.

Let
\[\fb:=\Omega_1\Omega_2.\]
If $\sigma\in \Autgr(A)$ is diagonal, then $\hdet(\sigma)
=\left(\det \sigma\restrict{A_1}\right)^2$ 
\cite[Theorem 1.5]{KK}. Using this fact, one can easily 
check that
\begin{align}\label{E6.3.1}
\sigma(\fb)= \hdet(\sigma) \fb.
\end{align}
This equation should be compared with \eqref{E1.4.2}.

\begin{proposition}\label{zzpro6.4}
Let $A=A(\alpha,\beta)$ be a PI noetherian graded down-up 
algebra. Then $A$ is Calabi--Yau if and only if $\fb \in Z(A)$.
\end{proposition}

\begin{proof}
Recall that $A$ is Calabi--Yau if and only if the Nakayama 
automorphism $\mu$ of $A$ is the identity, if and only if 
$\beta=-1$, if and only if $\omega_1 \omega_2=1$, if and 
only if $\Omega_1\Omega_2\in Z(A)$ by computation.
\end{proof}

It is easy to show, using these generating sets for $Z(A)$ 
given in \cite{kulk,zhaoDU}, that $Z(A)$ is regular when 
$(\alpha,\beta)=(0,1)$. We conjecture that this is the only 
case that $Z(A)$ is regular.

\begin{question}\label{xxque6.5}
Suppose $A=A(\alpha,\beta)$ be a PI noetherian graded down 
up algebra. For what parameters $(\alpha,\beta)$ is $Z(A)$ 
regular (resp. Gorenstein)?
\end{question}

The following should be compared with Theorem~\ref{zzthm0.11}.

\begin{proposition}\label{zzpro6.6}
Let $H$ be a finite subgroup of $\Autgr(A)$ consisting of 
diagonal automorphisms. The following are equivalent.
\begin{enumerate}
\item[(1)] 
The $H$-action on $A$ has trivial homological determinant.
\item[(2)] 
$A^H$ is Gorenstein.
\item[(3)] 
$\fb\in A^H$.
\end{enumerate}
\end{proposition}

\begin{proof}
(1) $\Leftrightarrow$ (2): 
By \cite[Proposition 6.4]{KKZ5}, $H$ does not contain any 
reflections. By \cite[Corollary 4.11]{KKZ4}, $A^H$ is 
Gorenstein if and only if $\hdet_H$ is trivial.

(1) $\Leftrightarrow$ (3): 
By \eqref{E6.3.1}, $\hdet_H$ is trivial if and only if 
$g(\fb)=\fb$ for all $g\in H$, if and only if $\fb\in A^H$.
\end{proof}

The following should be compared with Theorem~\ref{zzthm0.11}.

\begin{corollary}\label{zzcor6.7}
The following are equivalent.
\begin{enumerate}
\item[(1)] 
$A^{O}$ is Gorenstein where $O = \Oz(A)$.
\item[(2)] 
$\fb\in A^{O}$.
\end{enumerate}
\end{corollary}
\begin{proof}
By \cite{CGWZ2}, the ozone group $O$ consists of diagonal 
automorphisms. The assertion follows from 
Proposition~\ref{zzpro6.6} by setting $H=O$.
\end{proof}

\begin{question}\label{zzque6.8}
Since $Z \subseteq A^O$, then clearly $\fb\in Z$ implies 
that $\fb\in A^{O}$. Does the opposite implication hold?
\end{question}

\subsection*{Acknowledgments}
We thank Ken Goodearl and Colin Ingalls for their thoughts 
on material in this paper, and for questions we have included 
in the last section. 
The authors also thank the referee for several comments and suggestions that improved the paper.
R. Won was partially supported by an 
AMS--Simons Travel Grant and Simons Foundation grant \#961085. 
J.J. Zhang was partially supported by the US National Science 
Foundation (No. DMS-2302087 and DMS-2001015).

\bibliographystyle{plain}

\end{document}